\documentclass[11pt,a4paper]{article}
\usepackage{graphicx}
\usepackage{subcaption}
\usepackage{xcolor}
\usepackage{comment}
\usepackage[utf8]{inputenc}
\usepackage{comment}
\usepackage{authblk}

\usepackage{hyperref}
\hypersetup{
    colorlinks=true,
    linkcolor=blue,
    citecolor = magenta,
    filecolor=magenta,      
    urlcolor=cyan,
}

\usepackage{setspace,amsmath,amsthm,amsfonts,amssymb,mathtools,bm,algorithm,algpseudocode,
longtable,multirow,geometry,mathrsfs}
\usepackage[font=sf, labelfont={sf,bf}, margin=1cm]{caption}
\newtheorem{Proposition}{Proposition}
\newtheorem{Remark}{Remark}

\numberwithin{Theorem}{section}
\numberwithin{Definition}{section}
\numberwithin{Lemma}{section}
\numberwithin{Algorithm}{section}
\numberwithin{equation}{section}

\newtheorem{theorem}{Theorem}[section]

\newtheorem{lemma}[theorem]{Lemma} 
\newtheorem{assumption}{Assumption}
\newtheorem{premise}{Premise}

\newtheorem{remark}{Remark}
\renewcommand{\thealgorithm}{IP-PMM}

\makeatletter
\def\@cline#1-#2\@nil{%
  \omit
  \@multicnt#1%
  \advance\@multispan\m@ne
  \ifnum\@multicnt=\@ne\@firstofone{&\omit}\fi
  \@multicnt#2%
  \advance\@multicnt-#1%
  \advance\@multispan\@ne
  \leaders\hrule\@height\arrayrulewidth\hfill
  \cr
  \noalign{\nobreak\vskip-\arrayrulewidth}}
\makeatother
\begin{document}
\title{An Interior Point-Proximal Method of Multipliers for Positive Semi-Definite Programming}
\author[*]{Spyridon Pougkakiotis}
\author[*]{Jacek Gondzio}
\affil[*]{School of Mathematics, University of Edinburgh\newline \newline ERGO Technical Report 20--006}
\maketitle

\begin{abstract}
\par In this paper we generalize the Interior Point-Proximal Method of Multipliers (IP-PMM) presented in [\emph{An Interior Point-Proximal Method of Multipliers for Convex Quadratic Programming}, Computational Optimization and Applications, 78, 307--351 (2021)] for the solution of linear positive Semi-Definite Programming (SDP) problems, allowing inexactness in the solution of the associated Newton systems. In particular, we combine an infeasible Interior Point Method (IPM) with the Proximal Method of Multipliers (PMM) and interpret the algorithm (IP-PMM) as a primal-dual regularized IPM, suitable for solving SDP problems. We apply some iterations of an IPM to each sub-problem of the PMM until a satisfactory solution is found. We then update the PMM parameters, form a new IPM neighbourhood, and repeat this process. Given this framework, we prove polynomial complexity of the algorithm, under mild assumptions, and without requiring exact computations for the Newton directions. We furthermore provide a necessary condition for lack of strong duality, which can be used as a basis for constructing detection mechanisms for identifying pathological cases within IP-PMM. 
\end{abstract}

\section{Introduction}
\par Positive Semidefinite Programming (SDP) problems have attracted a lot of attention in the literature for more more than two decades, and have been used to model a plethora of different problems arising from control theory \cite[Chapter 14]{BalkaWang_BOOK_SPRINGER}, power systems \cite{LavaLow_IEEE_TRAN_POW_SYS}, stochastic optimization \cite{Ben-Tal_Nemi_MATH_OR}, truss optimization \cite{Weld_et_al_STRUCT_MULT_OPT}, and many other application areas (e.g. see \cite{BalkaWang_BOOK_SPRINGER,VandBoyd_APPL_NUM_MATH}). More recently, SDP has been extensively used for building tight convex relaxations of NP-hard combinatorial optimization problems (see \cite[Chapter 12]{BalkaWang_BOOK_SPRINGER}, and the references therein).
\par As a result of the seemingly unlimited applicability of SDP, numerous contributions have been made to optimization techniques 
suitable for solving such problems. The most remarkable milestone was achieved by Nesterov and Nemirovskii \cite{NestNemir_BOOK_SIAM}, who designed a polynomially convergent Interior Point Method (IPM) for the class of SDP problems. This led to the development of numerous successful IPM variants for SDP; some of theoretical (e.g. \cite{MizuJarr_MATH_PROG,Zhang_SIAM_J_OPT,ZhouToh_MATH_PROG}) and others of practical nature (e.g. \cite{Bella_Gondz_Porc_arxiv,Bella_Gondz_Porc_MATH_PROG,MOSEK}). While IPMs enjoy fast convergence, in theory and in practice, each IPM iteration requires the solution of a very large-scale linear system, even for small-scale SDP problems. What is worse, such linear systems are inherently ill-conditioned. A viable and successful alternative to IPMs for SDP problems (e.g. see \cite{ZhaoSunToh_SIAM_J_OPT}), which circumvents the issue of ill-conditioning without significantly compromising convergence speed, is based on the so-called Augmented Lagrangian method (ALM), which can be seen as the dual application of the proximal point method (as shown in \cite{ROCK_Math_OR}). The issue with ALMs is that, unlike IPMs, a consistent strategy for tuning the algorithm parameters is not known. Furthermore, polynomial complexity is lost, and is replaced with merely a finite termination. An IPM scheme combined with the Proximal Method of Multipliers (PMM) for solving SDP problems was proposed in \cite{Dehg_Goff_Orban_OMS}, and was interpreted as a primal-dual regularized IPM. The authors established global convergence, and numerically demonstrated the efficiency of the method. However, the latter is not guaranteed to converge to an $\epsilon$-optimal solution in a polynomial number of iterations, or even to find a global optimum in a finite number of steps. Finally, viable alternatives based on proximal splitting methods have been studied in \cite{JiangVander_OPT_ONLINE,Souto_Garcia_Veiga_OPT}. Such methods are very efficient and require significantly less computations and memory per iteration, as compared to IPM or ALM. However, as they are first-order methods, their convergence to high accuracy might be slow. Hence, such methodologies are only suitable for finding approximate solutions with low-accuracy. 
\par In this paper, we are extending the Interior Point-Proximal Method of Multipliers (IP-PMM) presented in \cite{Pougk_Gond_COAP}. In particular, the algorithm in \cite{Pougk_Gond_COAP} was developed for convex quadratic programming problems and assumed that the resulting linear systems are solved exactly. Under this framework, it was proved that IP-PMM converges in a polynomial number of iterations, under mild assumptions, and an infeasibility detection mechanism was established. An important feature of this method is that it provides a reliable tuning for the penalty parameters of the PMM; indeed, the reliability of the algorithm is established numerically in a wide variety of convex problems in \cite{Berga_et_al_NLAA,deSimone_et_al_arxiv,GondPougkPear_arxiv,Pougk_Gond_COAP}. In particular, the IP-PMMs proposed in \cite{Berga_et_al_NLAA,deSimone_et_al_arxiv,GondPougkPear_arxiv} use preconditioned iterative methods for the solution of the resulting linear systems, and are very robust despite the use of inexact Newton directions. In what follows, we develop and analyze an IP-PMM for linear SDP problems, which furthermore allows for inexactness in the solution of the linear systems that have to be solved at every iteration. We show that the method converges polynomially under standard assumptions. Subsequently, we provide a necessary condition for lack of strong duality, which can serve as a basis for constructing implementable detection mechanisms for pathological cases (following the developments in \cite{Pougk_Gond_COAP}). As is verified in \cite{Pougk_Gond_COAP}, IP-PMM is competitive with standard non-regularized IPM schemes, and is significantly more robust. This is because the introduction of regularization prevents severe ill-conditioning and rank deficiency of the associated linear systems solved within standard IPMs, which can hinder their convergence and numerical stability. For detailed discussions on the effectiveness of regularization within IPMs, the reader is referred to \cite{AltmanGondzio_OMS,ArmandBenoist_MATH_PROG,PougkGond_JOTA}, and the references therein. A particularly important benefit of using regularization, is that the resulting Newton systems can be preconditioned effectively (e.g. see the developments in \cite{Berga_et_al_NLAA,GondPougkPear_arxiv}), allowing for more efficient implementations, with significantly lowered memory requirements. We note that the paper is focused on the theoretical aspects of the method, and an efficient, scalable, and reliable implementation would require a separate study. Nevertheless, the practical effectiveness of IP-PMM (both in terms of efficiency, scalability, and robustness) has already been demonstrated for linear, convex quadratic \cite{Berga_et_al_NLAA,GondPougkPear_arxiv,Pougk_Gond_COAP}, and non-linear convex problems \cite{deSimone_et_al_arxiv}.
\par The rest of the paper is organized as follows. In  Section \ref{section preliminaries}, we provide some preliminary background and introduce our notation. Then, in Section \ref{section Algorithmic Framework}, we provide the algorithmic framework of the method. In Section \ref{section Polynomial Convergence}, we prove polynomial complexity of the algorithm, and establish its global convergence. In Section \ref{section Infeasible problems}, a necessary condition for lack of strong duality  is derived, and we discuss how it can be used to construct an implementable detection mechanism for pathological cases. Finally, we derive some conclusions in Section \ref{section conclusions}.
\section{Preliminaries and Notation} \label{section preliminaries}
\subsection{Primal-Dual Pair of SDP Problems}
\par Let the vector space $\mathcal{S}^n \coloneqq \{B \in \mathbb{R}^{n\times n} \colon B = B^\top \}$ be given, endowed with the inner product $\langle A, B \rangle = \textnormal{Tr}(AB)$, where $\textnormal{Tr}(\cdot)$ denotes the trace of a matrix. In this paper, we consider the following primal-dual pair of linear positive semi-definite programming problems, in the standard form:
\begin{equation} \label{non-regularized primal} \tag{P}
\underset{X \in \mathcal{S}^n}{\text{min}} \  \langle C,X\rangle , \ \ \text{s.t.}  \  \mathcal{A}X = b,   \ X \in \mathcal{S}^n_{+}, 
\end{equation}

\begin{equation} \label{non-regularized dual} \tag{D}
\underset{y \in \mathbb{R}^m,\ Z \in \mathcal{S}^n}{\text{max}}  \ b^\top y , \ \ \text{s.t.}\   \mathcal{A}^*y + Z = C,\ Z \in \mathcal{S}^n_{+},
\end{equation}
\noindent where $\mathcal{S}^{n}_{+} \coloneqq \{B \in \mathcal{S}^{n} \colon B \succeq 0\}$, $C,X,Z \in \mathcal{S}^{n}$, $b,y \in \mathbb{R}^m$, $\mathcal{A}$ is a linear operator on $\mathcal{S}^n$, $\mathcal{A}^*$ is the adjoint of $\mathcal{A}$, and $X\succeq 0$ denotes that $X$ is positive semi-definite. We note that the norm induced by the inner product $\langle A, B \rangle = \textnormal{Tr}(AB)$ is in fact the \textit{Frobenius norm}, denoted by $\|\cdot\|_{F}$. Furthermore, the adjoint $\mathcal{A}^* \colon \mathbb{R}^m \mapsto \mathcal{S}^n$ is such that $y^\top \mathcal{A}X = \langle \mathcal{A}^*y, X\rangle,\ \  \forall\ y \in \mathbb{R}^m, \ \forall\ X \in \mathcal{S}^n$.
\par For the rest of this paper, except for Section \ref{section Infeasible problems}, we will assume that the linear operator $\mathcal{A}$ is onto and that problems \eqref{non-regularized primal} and \eqref{non-regularized dual} are both strictly feasible (that is, Slater's constraint qualification holds for both problems). It is well-known that under the previous assumptions, the primal-dual pair \eqref{non-regularized primal}--\eqref{non-regularized dual} is guaranteed to have optimal solution for which strong duality holds (see \cite{NestNemir_BOOK_SIAM}). Such a solution can be found by solving the Karush--Kuhn--Tucker (KKT) optimality conditions for \eqref{non-regularized primal}--\eqref{non-regularized dual}, which read as follows:
\begin{equation} \label{non-regularized F.O.C}
\begin{bmatrix} 
\mathcal{A}^* y + Z -C\\
\mathcal{A} X - b\\
XZ
\end{bmatrix} = \begin{bmatrix}
0\\
0\\
0
\end{bmatrix},\qquad X,\ Z \in \mathcal{S}^n_+.
\end{equation}

\subsection{A Proximal Method of Multipliers}

\noindent The author in \cite{ROCK_Math_OR} presented for the first time the \textit{Proximal Method of Multipliers} (PMM), in order to solve general convex programming problems. Let us derive this method for the pair \eqref{non-regularized primal}--\eqref{non-regularized dual}. Given arbitrary starting point $(X_0,y_0) \in \mathcal{S}^n_{+}\times \mathbb{R}^m$, the PMM can be summarized by the following iteration:
\begin{equation} \label{PMM sub-problem}
\begin{split}
X_{k+1} =&\ \underset{X \in \mathcal{S}^n_+}{\arg\min}\bigg\{\langle C,X\rangle - y_k^\top (\mathcal{A}X - b) + \frac{\mu_k}{2}\|X-X_k\|_F^2 + \frac{1}{2\mu_k}\|AX-b\|_2^2 \bigg\},\\
y_{k+1} = &\ y_k - \frac{1}{\mu_k} (AX_{k+1} - b),
\end{split}
\end{equation}

\noindent where $\mu_k$ is a positive penalty parameter. The previous iteration admits a unique solution, for all $k$.
\par We can write \eqref{PMM sub-problem} equivalently by making use of the \textit{maximal monotone} operator $T_{\mathcal{L}} \colon \mathbb{R}^m\times \mathcal{S}^n  \rightrightarrows \mathbb{R}^m\times \mathcal{S}^n$ (see \cite{ROCK_Math_OR,ROCK_SIAM_J_CONTROL_OPT}), whose graph is defined as:

\begin{equation} \label{Primal Dual Maximal Monotone Operator}
T_{\mathcal{L}}(X,y) \coloneqq \{(V,u): V \in C - \mathcal{A}^*y + \partial \delta_{S^n_{+}}(X),\ u = \mathcal{A}X-b \},
\end{equation}
\noindent where $\delta_{S^n_{+}}(\cdot)$ is an indicator function defined as:
\begin{equation} \label{Indicator function}
\delta_{S^n_{+}}(X) \coloneqq 
     \begin{cases}
       0, &\quad\text{if } X \in \mathcal{S}^n_+, \\
       \infty, &\quad\text{otherwise,} \\ 
     \end{cases}
\end{equation}
\noindent and $\partial(\cdot)$ denotes the sub-differential of a function, hence (from \cite[Corollary 23.5.4]{Rockafellar_BOOK_PRINCETON}): 
\begin{equation*}
Z \in \partial \delta_{S^n_{+}}(X) \Leftrightarrow -Z \in \mathcal{S}^n_+,\ \langle X,Z\rangle = 0.
\end{equation*}
\noindent By convention, we have that $\partial \delta_{\mathcal{S}_+^n(X^*)} = \emptyset$ if $X^* \notin \mathcal{S}_+^n$. Given a bounded pair $(X^*,y^*)$ such that $(0,0) \in T_{\mathcal{L}}(X^*,y^*)$, we can retrieve a matrix $Z^* \in \partial \delta_{S^n_{+}}(X^*)$, using which $(X^*,y^*,-Z^*)$ is an optimal solution for \eqref{non-regularized primal}--\eqref{non-regularized dual}. By defining the \textit{proximal operator}:
\begin{equation} \label{Primal Dual Proximal Operator}
\mathcal{P}_k \coloneqq \bigg(I_{n+m} + \frac{1}{\mu_k}T_{\mathcal{L}}\bigg)^{-1},
\end{equation}
\noindent where $I_{n+m}$ is the identity operator of size $n+m$, and describes the direct sum of the idenity operators of $\mathcal{S}_n$ and $\mathbb{R}^m$, we can express \eqref{PMM sub-problem} as: 
\begin{equation} \label{PMM Operator Subproblem}
(X_{k+1},y_{k+1}) = \mathcal{P}_k(X_k,y_k),
\end{equation}
\noindent and it can be shown that $\mathcal{P}_k$ is single valued and firmly \textit{non-expansive} (see \cite{ROCK_SIAM_J_CONTROL_OPT}).
\subsection{An Infeasible Interior Point Method}
\par In what follows we present a basic infeasible IPM suitable for solving the primal-dual pair \eqref{non-regularized primal}--\eqref{non-regularized dual}. Such methods handle the conic constraints by introducing a suitable logarithmic barrier in the objective (for an extensive study of logarithmic barriers, the reader is referred to \cite{NestNemir_BOOK_SIAM}). At each iteration, we choose a \textit{barrier parameter} $\mu > 0$ and form the logarithmic \textit{barrier primal-dual pair:}

\begin{equation} \label{non-regularized barrier primal} 
\underset{X \in \mathcal{S}^n}{\text{min}} \  \langle C,X\rangle - \mu \ln(\det(X)), \ \ \text{s.t.}  \  \mathcal{A}X = b,  
\end{equation}

\begin{equation} \label{non-regularized barrier dual} 
\underset{y \in \mathbb{R}^m,\ Z \in \mathcal{S}^n}{\text{max}}  \ b^\top y + \mu \ln(\det(Z)) , \ \ \text{s.t.}\   \mathcal{A}^*y + Z = C.
\end{equation}
\noindent The first-order (barrier) optimality conditions of \eqref{non-regularized barrier primal}--\eqref{non-regularized barrier dual} read as follows:
\begin{equation} \label{non-regularized barrier F.O.C}
\begin{bmatrix} 
\mathcal{A}^* y + Z -C\\
\mathcal{A} X - b\\
XZ - \mu I_n
\end{bmatrix} = \begin{bmatrix}
0\\
0\\
0
\end{bmatrix},\qquad X,\ Z \in \mathcal{S}^n_{++},
\end{equation}
\noindent where $\mathcal{S}^n_{++} \coloneqq \{B \in \mathcal{S}^n: B\succ 0\}$. For every chosen value of $\mu$, we want to approximately solve the following non-linear system of equations:
\begin{equation*} 
F_{\sigma,\mu}^{IPM}(w) \coloneqq \begin{bmatrix}
\mathcal{A}^* y + Z -C\\
\mathcal{A} X - b\\
XZ - \sigma \mu I_n
\end{bmatrix} = \begin{bmatrix}
0\\
0\\
0
\end{bmatrix},
\end{equation*}
\noindent where, with a slight abuse of notation, we set $w = (X,y,Z)$. Notice that $F_{\sigma,\mu}^{IPM}(w) = 0$ is a perturbed form of the barrier optimality conditions. In particular, $\sigma \in (0,1)$ is a \textit{centering parameter} which determines how fast $\mu$ will be forced to decrease at the next IPM iteration. For $\sigma = 1$ we recover the barrier optimality conditions in \eqref{non-regularized barrier F.O.C}, while for $\sigma = 0$  we recover the optimality conditions in \eqref{non-regularized F.O.C}.
\par In IPM literature it is common to apply Newton method to solve approximately the system of non-linear equations $F_{\sigma,\mu}^{IPM}(w) = 0$. Newton method is favored for systems of this form due to the \textit{self-concordance} of the logarithmic barrier (see \cite{NestNemir_BOOK_SIAM}). However, a well-known issue in the literature is that the matrix $XZ$ is not necessarily symmetric. A common approach to tackle this issue is to employ a symmetrization operator $H_P : \mathbb{R}^{n\times n} \mapsto \mathcal{S}^n$, such that $H_P(XZ) = \mu I$ if and only if $XZ = \mu I$, given that $X,\ Z \in \mathcal{S}_+^n$. Following Zhang (\cite{Zhang_SIAM_J_OPT}), we employ the following operator: $H_P : \mathbb{R}^{n\times n} \mapsto \mathcal{S}^n$:
\begin{equation} \label{Symmetrization operator}
H_P(B) \coloneqq \frac{1}{2}(PBP^{-1} + (PBP^{-1})^\top ),
\end{equation}
\noindent where $P$ is a non-singular matrix. It can be shown that the central path (a key notion used in IPMs--see \cite{NestNemir_BOOK_SIAM}) can be equivalently defined as $H_P(XZ) = \mu I$, for any non-singular $P$. In this paper, we will make use of the choice $P_k = Z_k^{-\frac{1}{2}}$. For a plethora of alternative choices, the reader is referred to \cite{Todd_OTP_METH_SOFT}. We should note that the analysis in this paper can be tailored to different symmetrization strategies, and this choice is made for simplicity of exposition.  \par At the beginning of the $k$-th iteration, we have $w_k = (X_k, y_k, Z_k)$ and $\mu_k$ available. The latter is defined as $\mu_k = \frac{\langle X_k, Z_k \rangle}{n}$. By substituting the symmetrized complementarity in the last block equation and applying Newton method, we obtain the following system of equations:

\begin{equation} \label{non-regularized Newton system}
\begin{bmatrix}
0 & \mathcal{A}^* & I_n\\
\mathcal{A} & 0 & 0 \\
\mathcal{E}_k & 0 & \mathcal{F}_k
\end{bmatrix} \begin{bmatrix}
\Delta X\\
\Delta y\\
\Delta Z
\end{bmatrix} = \begin{bmatrix}
C - \mathcal{A}^*y - Z_k\\
b - \mathcal{A}X_k\\
\mu I_n - H_{P_k}(X_kZ_k)
\end{bmatrix},
\end{equation} 
\noindent where $\mathcal{E}_k \coloneqq \nabla_X H_{P_k}(X_kZ_k)$, and $\mathcal{F}_k \coloneqq \nabla_Z H_{P_k}(X_kZ_k)$. 

\subsection{Vectorized Format}
\par In what follows we vectorize the associated operators, in order to work with matrices. In particular, given any matrix $B \in \mathbb{R}^{m \times n}$, we denote its vectorized form as $\bm{B}$, which is a vector of size $mn$, obtained by stacking the columns of $B$, from the first to the last. For the rest of this manuscript, any boldface letter denotes a vectorized matrix. Furthermore, if $\mathcal{A}: \mathcal{S}^n \mapsto \mathbb{R}^m$ is a linear operator, we can define it component-wise as $(\mathcal{A}X)_i \coloneqq \langle A_i , X \rangle$, for $i = 1,\ldots,m$, and any $X \in \mathcal{S}^n$, where $A_i \in \mathcal{S}^n$. Furthermore, the adjoint of this operator, that is $\mathcal{A}^*: \mathbb{R}^m \mapsto \mathcal{S}^n$ is defined as $\mathcal{A}^*y \coloneqq \sum_{i = 1}^{m} y_i A_i$, for all $y \in \mathbb{R}^m$. Using this notation, we can equivalently write \eqref{non-regularized primal}--\eqref{non-regularized dual} in the following form:
\begin{equation} \label{non-regularized vectorized primal} 
\underset{X \in \mathcal{S}^n}{\text{min}} \  \langle C,X\rangle , \qquad \text{s.t.}  \  \langle A_i, X \rangle = b_i, \quad i = 1,\ldots,m, \qquad \ X \in \mathcal{S}^n_{+}, 
\end{equation}

\begin{equation} \label{non-regularized vectorized dual} 
\underset{y \in \mathbb{R}^m,\ Z \in \mathcal{S}^n}{\text{max}}  \  \ b^\top y , \qquad
 \text{s.t.}  \   \sum_{i=1}^m y_i A_i + Z = C, \qquad \ Z \in \mathcal{S}^n_{+}.
\end{equation}
\noindent The first-order optimality conditions can be re-written as:
\begin{equation*}
\begin{bmatrix}
A^\top y + \bm{Z} - \bm{C} \\
A\bm{X} - b\\
XZ
\end{bmatrix} = \begin{bmatrix}
0\\
0\\
0
\end{bmatrix}, \qquad X,\ Z \in \mathcal{S}^n_+,
\end{equation*}
\noindent where $A^\top  = [\bm{A}_1\ \bm{A}_2\ \cdots \ \bm{A}_m]$. 
\par Below we summarize any additional notation that is used later in the paper. An iteration of the algorithm is denoted by $k \in \mathbb{N}$. Given an arbitrary matrix $A$ (resp., vector $x$), $A_k$ (resp., $x_k$) denotes that the matrix (resp., vector) depends on the iteration $k$. An optimal solution to the pair \eqref{non-regularized primal}--\eqref{non-regularized dual} will be denoted as $(X^*,y^*,Z^*)$. Optimal solutions of different primal-dual pairs will be denoted using an appropriate subscript, in order to distinguish them (e.g. $(X_r^*,y_r^*,Z_r^*)$ will denote an optimal solution for a PMM sub-problem). Any norm (resp., semi-norm) is denoted by $\| \cdot \|_{\chi}$, where $\chi$ is used to distinguish between different norms (e.g. $\|\cdot\|_2$ denotes the Euclidean norm). Given two matrices $X,\ Y \in \mathcal{S}^n_+$, we write $X \succeq Y$ when $X$ is larger than $Y$ with respect to the Loewner ordering.   Given two logical statements $T_1,\ T_2$, the condition $T_1 \wedge T_2$ is true only when both $T_1$ and $T_2$ are true.  Given two real-valued positive increasing functions $T(\cdot)$ and $f(\cdot)$, we say that $T(x) = O(f(x))$ (resp., $T(x) = \Omega(f(x))$) if there exist $x_0\geq 0,\ c_1 > 0$, such that $T(x) \leq c_1 f(x)$ (resp., $c_2 > 0$ such that $T(x) \geq c_2 f(x)$), for all $x \geq x_0$. We write $T(x) = \Theta(f(x))$ if and only if $T(x) = O(f(x))$ and $T(x) = \Omega(f(x))$. Finally, let an arbitrary matrix $A$ be given. The maximum (resp., minimum) singular value of $A$ is denoted by $\eta_{\max}(A)$ (resp., $\eta_{\min}(A)$). Similarly, the maximum (resp., minimum) eigenvalue of a square matrix $A$ is denoted by $\nu_{\max}(A)$ (resp., $\nu_{\min}(A)$).
\section{An Interior Point-Proximal Method of Multipliers for SDP} \label{section Algorithmic Framework}
\par In this section we present an inexact extension of IP-PMM presented in \cite{Pougk_Gond_COAP}, suitable for solving problems of the form of \eqref{non-regularized primal}--\eqref{non-regularized dual}. Assume that we have available an estimate $\lambda_k$ for a Lagrange multiplier vector at iteration $k$.  Similarly, denote by $\Xi_k \in \mathcal{S}_+^n$ an estimate of a primal solution. As we discuss later, these estimate sequences (i.e. $\{\lambda_k\}, \{\Xi_k\}$) are produced by the algorithm, and represent the dual and primal proximal estimates, respectively. During the $k$-th iteration of the PMM, applied to \eqref{non-regularized primal}, the following proximal penalty function has to be minimized:
\begin{equation} \label{PMM lagrangian}
\begin{split}
\mathcal{L}^{PMM}_{\mu_k} (X;\Xi_k, \lambda_k) \coloneqq \langle C, X \rangle  -\lambda_k^\top (\mathcal{A}X - b) + \frac{1}{2\mu_k}\|\mathcal{A}X-b\|_{2}^2 + \frac{\mu_k}{2}\|X-\Xi_k\|_{F}^2,
\end{split}
\end{equation}
\noindent with $\mu_k > 0$ being some non-increasing penalty parameter. Notice that this is equivalent to the iteration  (\ref{PMM sub-problem}).  We approximately minimize \eqref{PMM lagrangian} by applying one (or a few) iterations of the previously presented infeasible IPM. We alter \eqref{PMM lagrangian} by adding a logarithmic barrier:
\begin{equation} \label{Proximal IPM Penalty}
\mathcal{L}^{IP-PMM}_{\mu_k} (X;\Xi_k, \lambda_k) \coloneqq \mathcal{L}^{PMM}_{\mu_k} (X;\Xi_k, \lambda_k) - \mu_k \log(\det(X)),
\end{equation}
\noindent and we treat $\mu_k$ as the barrier parameter. In order to form the optimality conditions of this sub-problem, we equate the gradient of $\mathcal{L}^{IP-PMM}_{\mu_k}(\cdot;\Xi_k,\lambda_k)$ to the zero vector, i.e.:
\begin{equation*}
C  - \mathcal{A}^* \lambda_k + \frac{1}{\mu_k}\mathcal{A}^*(\mathcal{A}X - b) + \mu_k (X - \Xi_k) - \mu_k X^{-1} = 0.
\end{equation*}
\par Introducing the variables $y = \lambda_k - \frac{1}{\mu_k}(\mathcal{A}X - b)$ and $Z = \mu_k X^{-1}$, yields:
\begin{equation} \label{non-vectorized Proximal IPM FOC}
\begin{split}
\begin{bmatrix}
C - \mathcal{A}^*y - Z + \mu_k(X-\Xi_k)\\
\mathcal{A}X + \mu_k (y - \lambda_k) - b\\
XZ -  \mu_k I_n
\end{bmatrix} = \begin{bmatrix}
0\\
0\\
0
\end{bmatrix} \Leftrightarrow \begin{bmatrix}
\bm{C}- A^\top y - \bm{Z} + \mu_k(\bm{X}-\bm{\Xi}_k)\\
A\bm{X}+ \mu_k (y - \lambda_k) - b\\
H_{P_k}(XZ) -  \mu_k I_n
\end{bmatrix} = \begin{bmatrix}
0\\
0\\
0
\end{bmatrix},
\end{split}
\end{equation}
\noindent where the second system is obtained by introducing the symmetrization in \eqref{Symmetrization operator}, and by vectorizing the associated matrices and operators.
\par Given an arbitrary vector $b \in \mathbb{R}^m$, and matrix $C \in \mathbb{R}^{n \times n}$, we define the semi-norm:
\begin{equation} \label{semi-norm definition}
\|(b,\bm{C})\|_{\mathcal{S}} \coloneqq \min_{X,y,Z}\bigg\{\|(\bm{X},\bm{Z})\|_2\ :\begin{matrix} A\bm{X} = b, \\ A^\top y + \bm{Z} = \bm{C}\end{matrix}\bigg\}.
\end{equation}
\noindent A similar semi-norm was used before in \cite{MizuJarr_MATH_PROG}, as a way to measure infeasibility for the case of linear programming problems. For a discussion of the properties of the aforementioned semi-norm, as well as how to evaluate it (using an appropriate QR factorization, which can be computed in a polynomial time), the reader is referred to \cite[Section 4]{MizuJarr_MATH_PROG}.
\paragraph{Starting Point.}
\par  Let us define the starting point for IP-PMM. For that, we set $(X_0,Z_0) = \rho(I_n,I_n)$, for some $\rho > 0$. We also set $y_0$ to some arbitrary value (e.g. $y_0 = 0$), and $\mu_0 = \frac{\langle X_0, Z_0 \rangle}{n}$. Using the aforementioned triple, we have:
\begin{equation} \label{starting point}
A\bm{X}_0 = b + \bar{b},\ A^\top y_0 + \bm{Z}_0 = \bm{C} + \bm{\bar{C}},\  \Xi_0 = X_0,\ \lambda_0 = y_0.
\end{equation}
\noindent for some $\bar{b} \in \mathbb{R}^m$, and $\bar{C} \in \mathcal{S}^{n}$. 
\paragraph{Neighbourhood.}
\par Below, we describe a neighbourhood in which the iterations of the method should lie. Unlike most path-following methods, we have to define a family of neighbourhoods that depend on the PMM sub-problem parameters. 
\par Given \eqref{starting point}, some $\mu_k$, $\lambda_k$, and $\Xi_k$, we define the \textit{regularized set of centers}:
\begin{equation*}
\mathscr{P}_{\mu_k}(\Xi_k,\lambda_k) \coloneqq \large\{(X,y,Z)\in \mathcal{C}_{\mu_k}(\Xi_k,\lambda_k)\ :\ X \in \mathcal{S}^n_{++},\ Z \in \mathcal{S}^n_{++},\ XZ = \mu_k I_n \large\},
\end{equation*}
\begin{equation*}
\mathscr{C}_{\mu_k}(\Xi_k,\lambda_k) \coloneqq \bigg\{(X,y,Z)\ :\quad \begin{matrix}
A\bm{X} + \mu_k (y-\lambda_k) = b + \frac{\mu_k}{\mu_0} \bar{b},\\
 A^\top y + \bm{Z} - \mu_k(\bm{X}- \bm{\Xi}_k) = \bm{C} +  \frac{\mu_k}{\mu_0}\bm{\bar{C}}
\end{matrix} \bigg\},
\end{equation*}
\noindent where $\bar{b},\ \bar{C}$ are as in \eqref{starting point}. The term set of centers originates from \cite{MizuJarr_MATH_PROG}.
\par We enlarge the previous set, by defining the following set:
\begin{equation*}
\begin{split}
\tilde{\mathscr{C}}_{\mu_k}(\Xi_k,\lambda_k) \coloneqq  \Bigg\{(X,y,Z)\ :\quad \begin{matrix}
A\bm{X} + \mu_k(y-\lambda_k) = b + \frac{\mu_k}{\mu_0} (\bar{b}+\tilde{b}_{k}),\\
 A^\top y + \bm{Z} - \mu_k (\bm{X}- \bm{\Xi}_k) =\bm{C} +  \frac{\mu_k}{\mu_0}(\bm{\bar{C}}+\bm{\tilde{C}}_{k})\\
\|(\tilde{b}_{k},\bm{\tilde{C}}_{k})\|_2 \leq K_N,\ \|(\tilde{b}_{k},\bm{\tilde{C}}_{k})\|_{\mathcal{S}} \leq \gamma_{\mathcal{S}} \rho
\end{matrix} \Bigg\},
\end{split}
\end{equation*}
where $K_N > 0$ is a constant, $\gamma_{\mathcal{S}} \in (0,1)$ and $\rho>0$ is as defined in the starting point.  The vector $\tilde{b}_{k}$ and the matrix $\tilde{C}_{k}$ represent the current scaled (by $\frac{\mu_0}{\mu_k}$) infeasibilities, and will vary depending on the iteration $k$. While these can be defined recursively, it is not necessary. Instead it suffices to know that they satisfy the bounds given in the definition of the previous set. In essence, the previous set requires these scaled infeasibilities to be bounded above by some constants, with respect to the 2-norm as well as the semi-norm defined in \eqref{semi-norm definition}. We can now define a family of neighbourhoods:
\begin{equation} \label{Small neighbourhood}
\mathscr{N}_{\mu_k}(\Xi_k,\lambda_k) \coloneqq \bigg\{(X,y,Z) \in \tilde{\mathscr{C}}_{\mu_k}(\Xi_k,\lambda_k)\ :\quad \begin{matrix}\ X \in \mathcal{S}^n_{++},\ Z \in \mathcal{S}^n_{++},\\ \|H_P(XZ) - \mu I_n\|_F \leq \gamma_{\mu} \mu_k \end{matrix}\bigg\},
\end{equation}
\noindent where $\gamma_{\mu} \in (0,1)$ is a constant restricting the symmetrized complementarity products. Obviously, the starting point defined in \eqref{starting point} belongs to the neighbourhood $\mathscr{N}_{\mu_0}(\Xi_0,\lambda_0)$, with $(\tilde{b}_{0},\bm{\tilde{C}}_{0}) = (0,0)$.  Notice that the neighbourhood depends on the choice of the constants $K_N$, $\gamma_{\mathcal{S}}$, $\gamma_{\mu}$. However, as the neighbourhood also depends on the parameters $\mu_k,\ \lambda_k,\ \Xi_k$, the dependence on the constants is omitted for simplicity of notation. 
\paragraph{Newton System.}
\par As discussed earlier, we employ the Newton method for approximately solving a perturbed form of system \eqref{non-vectorized Proximal IPM FOC}, for all $k$. In particular, we perturb \eqref{non-vectorized Proximal IPM FOC} in order to take into consideration the target reduction of the barrier parameter $\mu_k$ (by introducing the centering parameter $\sigma_k$), as well as to incorporate the initial infeasibility, given our starting point in \eqref{starting point}. In particular, we would like to approximately solve the following system:
\begin{equation} \label{exact non-vectorized Newton System}
\begin{split}
	\begin{bmatrix} 
		-(\mu_k I_n) &   \mathcal{A^*} & I_n\\
		\mathcal{A} & \mu_k I_m & 0\\
		Z_k &  0  &X_k 
	\end{bmatrix} 
	\begin{bmatrix}
		\Delta X_k\\
		\Delta y_k\\
		\Delta Z_k
	\end{bmatrix}
	=  
	\begin{bmatrix}
		(C + \frac{\sigma_k \mu_k}{\mu_0}\bar{C})  - \mathcal{A}^* y_k - Z_k +\sigma_k\mu_k (X_k - \Xi_k)\\
		-\mathcal{A}X_k  -\sigma_k\mu_k (y_k - \lambda_k)+ (b +\frac{\sigma_k \mu_k}{\mu_0}\bar{b}) \\
 		-X_kZ_k + \sigma_{k} \mu_k I_n
	\end{bmatrix},
\end{split}
\end{equation}
\noindent where $\bar{b},\ \bar{C}$ are as in \eqref{starting point}. We note that we could either first linearize the last block equation of \eqref{non-vectorized Proximal IPM FOC} and then apply the symmetrization, defined in \eqref{Symmetrization operator}, or first apply the symmetrization directly to the last block equation of \eqref{non-vectorized Proximal IPM FOC} and then linearize it. Both approaches are equivalent. Hence, following the former approach, we obtain the vectorized Newton system, that has to be solved at every iteration of IP-PMM:
\begin{equation} \label{inexact vectorized Newton System}
\begin{split}
	\begin{bmatrix} 
		-(\mu_k I_{n^2}) &   A^\top  & I_{n^2}\\
		A & \mu_k I_m & 0\\
		E_k &  0  &F_k 
	\end{bmatrix} 
	\begin{bmatrix}
		\bm{\Delta X}_k\\
		\Delta y_k\\
		\bm{\Delta Z}_k
	\end{bmatrix}
	= \\ \begin{bmatrix}
		(\bm{C} + \frac{\sigma_k \mu_k}{\mu_0}\bm{\bar{C}})  -A^\top  y_k - \bm{Z}_k +\sigma_k\mu_k (\bm{X}_k - \bm{\Xi}_k)\\
		-A \bm{X}_k  -\sigma_k\mu_k (y_k - \lambda_k)+ (b +\frac{\sigma_k \mu_k}{\mu_0}\bar{b}) \\
 		-(Z_k^{\frac{1}{2}}\otimes Z_k^{\frac{1}{2}})\bm{X}_k + \sigma_{k} \mu_k \bm{I}_{n}
	\end{bmatrix} + \begin{bmatrix}
	\bm{\mathsf{E}}_{d,k}\\
	\epsilon_{p,k}\\
	\bm{\mathsf{E}}_{\mu,k}
	\end{bmatrix},
\end{split}
\end{equation}
\noindent where $E_k = (Z_k^{\frac{1}{2}} \otimes Z_k^{\frac{1}{2}})$, $F_k = \frac{1}{2}\big(Z_k^{\frac{1}{2}}X_k \otimes Z_k^{-\frac{1}{2}} + Z^{-\frac{1}{2}} \otimes Z_k^{\frac{1}{2}}X_k \big)$, and $(\mathsf{E}_{d,k},\epsilon_{p,k},\mathsf{E}_{\mu,k})$ models potential errors, occurring by solving the symmetrized version of system \eqref{exact non-vectorized Newton System} inexactly (e.g. by using a Krylov subspace method). In order to make sure that the computed direction is accurate enough, we impose the following accuracy conditions:
\begin{equation} \label{Krylov method termination conditions}
\|\bm{\mathsf{E}}_{\mu,k}\|_2 = 0,\qquad \|(\epsilon_{p,k},\bm{\mathsf{E}}_{d,k})\|_2 \leq \frac{\sigma_{\min}}{4\mu_0} K_N \mu_k,\qquad \|(\epsilon_{p,k},\bm{\mathsf{E}_{d,k}})\|_{\mathcal{S}} \leq \frac{\sigma_{\min}}{4\mu_0} \gamma_{\mathcal{S}}\rho \mu_k,
\end{equation}
\noindent where $\sigma_{\min}$ is the minimum allowed value for $\sigma_k$, $K_N,\ \gamma_{\mathcal{S}}$ are constants defined by the neighbourhood in \eqref{Small neighbourhood}, and $\rho$ is defined in the starting point in \eqref{starting point}. Notice that the condition $\|\bm{\mathsf{E}}_{\mu,k}\|_2 = 0$ is imposed without loss of generality, since it can be easily satisfied in practice. For more on this, see the discussion in \cite[Section 3]{ZhouToh_MATH_PROG} and \cite[Lemma 4.1]{Gu_SIAM_J_OPT}. Furthermore, as we will observe in Section \ref{section Polynomial Convergence}, the bound on the error with respect to the semi-norm defined in \eqref{semi-norm definition} is required to ensure polynomial complexity of the method. While evaluating this semi-norm is not particularly practical (and is never evaluated in practice, e.g. see \cite{Berga_et_al_NLAA,deSimone_et_al_arxiv,Pougk_Gond_COAP}), it can be done in a polynomial time (see \cite[Section 4]{MizuJarr_MATH_PROG}), and hence does not affect the polynomial nature of the algorithm. The algorithmic scheme of the method is summarized in Algorithm \ref{Algorithm PMM-IPM}.
\renewcommand{\thealgorithm}{IP--PMM}
\begin{algorithm}[!ht]
\caption{Interior Point-Proximal Method of Multipliers}
    \label{Algorithm PMM-IPM}
    \textbf{Input:}  $\mathcal{A}, b, C$, $\text{tol}$.\\
    \textbf{Parameters:} $0< \sigma_{\min} \leq \sigma_{\max} \leq 0.5$, $K_N > 0$, $0<\gamma_{\mathcal{S}} < 1,\ 0<\gamma_{\mu} < 1$.\\
    \textbf{Starting point:} Set as in \eqref{starting point}.
\begin{algorithmic}
\For {($k= 0,1,2,\cdots$)}
\If {$\Bigg(\bigg(\|A\bm{X}_k - b\|_2 < \text{tol}\bigg) \wedge \bigg(\|\bm{C} - A^\top  y_k - \bm{Z}_k\|_2 < \text{tol}\bigg) \wedge \bigg(\frac{\langle X_k, Z_k \rangle}{n} < \text{tol}\bigg)\Bigg)$}
	\State \Return $(X_k,y_k,Z_k)$.	
\Else
	\State Choose $\sigma_k \in [\sigma_{\min},\sigma_{\max}]$ and solve \eqref{inexact vectorized Newton System} so that \eqref{Krylov method termination conditions} holds.
	\State Choose $\alpha_k$, as the largest $\alpha \in (0,1]$, s.t.  $\mu_k(\alpha)  \leq  \ (1-0.01 \alpha)\mu_k$, and:
	\begin{equation*}
	\begin{split}
	 (X_k +  \alpha_k  & \Delta X_k, y_k +  \alpha_k \Delta y_k, Z_k + \alpha_k \Delta Z_k) \in  \ \mathscr{N}_{\mu_k(\alpha)}(\Xi_k,\lambda_k),\\ 
	\text{where, }\ \ & \mu_{k}(\alpha) = \frac{\langle X_k + \alpha_k \Delta X_k,Z_k + \alpha_k \Delta Z_k\rangle}{n}.
	\end{split}
	\end{equation*}	
	\State Set $(X_{k+1},y_{k+1},Z_{k+1}) = (X_k + \alpha_k \Delta X_k, y_k + \alpha_k \Delta y_k, Z_k + \alpha_k \Delta Z_k)$.
	\State Set $\mu_{k+1} = \frac{\langle X_{k+1},Z_{k+1}\rangle}{n}$.
	\State Let $r_p = A\bm{X}_{k+1} - (b + \frac{\mu_{k+1}}{\mu_0}\bar{b})$, $\bm{R}_d = (\bm{C} + \frac{\mu_{k+1}}{\mu_0}\bm{\bar{C}})- A^\top y_{k+1} -\bm{Z}_{k+1}.$
  	\If {\bigg(\big($\|(r_p,\bm{R}_d)\|_2 \leq K_N \frac{\mu_{k+1}}{\mu_0} \big) \wedge \big(\|(r_p,\bm{R}_d)\|_{\mathcal{S}} \leq \gamma_{\mathcal{S}}\rho \frac{\mu_{k+1}}{\mu_0}\big)$\bigg)}
  		\State $(\Xi_{k+1},\lambda_{k+1}) = (X_{k+1},y_{k+1})$.
  	\Else
  		\State $(\Xi_{k+1},\lambda_{k+1}) = (\Xi_{k},\lambda_{k})$.
  	\EndIf
\EndIf
\EndFor 
\end{algorithmic}
\end{algorithm}
\par Algorithm \ref{Algorithm PMM-IPM} deviates from standard IPM schemes due to the solution of a different Newton system, as well as due to the possible updates of the proximal estimates, i.e. $\Xi_k$ and $\lambda_k$. Notice that when these estimates are updated, the neighbourhood in \eqref{Small neighbourhood} changes as well, since it is parametrized by them. Intuitively, when this happens, the algorithm accepts the current iterate as a sufficiently accurate solution to the associated PMM sub-problem. However, as we will see in Section \ref{section Polynomial Convergence}, it is not necessary for these estimates to converge to a primal-dual solution, for Algorithm \ref{Algorithm PMM-IPM} to converge. Instead, it suffices to ensure that these estimates will remain bounded. In light of this, Algorithm \ref{Algorithm PMM-IPM} is not studied as an inner-outer scheme, but rather as a standard IPM scheme. We will return to this point at the end of Section \ref{section Polynomial Convergence}.
\section{Convergence Analysis} \label{section Polynomial Convergence}

\par In this section we prove polynomial complexity of Algorithm \ref{Algorithm PMM-IPM}, and establish its global convergence. The analysis is modeled after that in \cite{Pougk_Gond_COAP}. We make use of the following two standard assumptions, generalizing those employed in \cite{Pougk_Gond_COAP} to the SDP case.
\begin{assumption} \label{Assumption 1}
The problems \eqref{non-regularized primal} and \eqref{non-regularized dual} are strictly feasible, that is, Slater's constraint qualification holds for both problems. Furthermore, there exists an optimal solution $(X^*, y^*, Z^*)$ and a constant $K_* > 0$ independent of $n$ and $m$ such that $\|(\bm{X}^*,y^*,\bm{Z}^*)\|_F \leq K_* \sqrt{n}$.
\end{assumption}

\begin{assumption} \label{Assumption 2}
The vectorized constraint matrix $A$ of \textnormal{\eqref{non-regularized primal}} has full row rank, that is $\textnormal{rank}(A) = m$. Moreover, there exist constants $K_{A,1} > 0$, $K_{A,2} > 0$,  $K_{r,1} >0 $, and $\ K_{r,2} > 0$, independent of $n$ and $m$, such that:
\begin{equation*}
\eta_{\min}(A) \geq K_{A,1},\quad \eta_{\max}(A) \leq K_{A,2},\quad \|b\|_{\infty}\leq K_{r,1},\quad \|C\|_2 \leq K_{r,2} \sqrt{n}.
\end{equation*}
\end{assumption}
\begin{Remark} \label{Remark 1 on assumptions}
\par Note that the independence of the previous constants from $n$ and $m$ is assumed for simplicity of exposition. In particular, as long as these constants depend polynomially on $n$ (or $m$), the analysis still holds, simply by altering the worst-case polynomial bound for the number of iterations of the algorithm (given later in Theorem \textnormal{\ref{Theorem complexity}}).
\end{Remark}
\begin{Remark} \label{Remark 2 on assumptions} Assumption \textnormal{\ref{Assumption 1}} is a direct extension of that in \textnormal{\cite[Assumption 1]{Pougk_Gond_COAP}}. Given the positive semi-definiteness of $X^*$ and $Z^*$, it implies that $\textnormal{Tr}(X^*) + \textnormal{Tr}(Z^*) \leq 2 K_* n$ (from equivalence of the norms $\|\cdot\|_1$ and $\|\cdot\|_2$), which is one of the assumptions employed in \textnormal{\cite{Zhang_SIAM_J_OPT,ZhouToh_MATH_PROG}}. Notice that we assume  $n > m$, without loss of generality. The theory in this section would hold if $m > n$, simply by replacing $n$ by $m$ in the upper bound of the norm of the optimal solution as well as of the problem data.
\end{Remark}
\par Before proceeding with the convergence analysis, we briefly provide an outline of it, for the convenience of the reader. Firstly, it should be noted that polynomial complexity as well as global convergence of Algorithm \ref{Algorithm PMM-IPM} is proven by induction on the iterations $k$ of the method. To that end, we provide some necessary technical results in Lemmas \ref{Lemma non-expansiveness}--\ref{Lemma tilde point}. Then, in Lemma \ref{Lemma boundedness of x z} we are able to show that the iterates $(X_k,y_k,Z_k)$ of Algorithm \ref{Algorithm PMM-IPM} will remain bounded for all $k$. Subsequently, we provide some additional technical results in Lemmas \ref{Auxiliary Lemma bound on scaled matrices}--\ref{Auxiliary Lemma scaled rhs of third block of Newton system}, which are then used in Lemma \ref{Lemma boundedness Dx Dz}, where we show that the Newton direction computed at every iteration $k$ is also bounded. All the previous are utilized in Lemmas \ref{Lemma step-length-part 1}--\ref{Lemma step-length-part 2}, where we provide a lower bound for the step-length $\alpha_k$ chosen by Algorithm \ref{Algorithm PMM-IPM} at every iteration $k$. Then, $Q$-linear convergence of $\mu_k$ (with $R$-linear convergence of the regularized residuals) is shown in Theorem \ref{Theorem mu convergence}. Polynomial complexity is proven in Theorem \ref{Theorem complexity}, and finally, global convergence is established in Theorem \ref{Theorem convergence for the feasible case}.
\par Let us now use the properties of the proximal operator defined in \eqref{Primal Dual Proximal Operator}. 

\begin{lemma} \label{Lemma non-expansiveness}
Given Assumption \textnormal{\ref{Assumption 1}}, and for all $\lambda \in \mathbb{R}^m$, $\Xi \in \mathcal{S}_+^n$ and $0 \leq \mu < \infty$, there exists  a unique pair $(X_r^*,y_r^*)$, such that $(X_r^*,y_r^*) = \mathcal{P}(\Xi,\lambda),$ $X_r^* \in \mathcal{S}_+^n$, and
\begin{equation} \label{non-expansiveness property}
\|(\bm{X}_r^*,y_r^*)-(\bm{X}^*,y^*)\|_{2} \leq \|(\bm{\Xi},\lambda)-(\bm{X}^*,y^*)\|_{2},
\end{equation}
\noindent where $\mathcal{P}(\cdot)$ is defined as in \eqref{Primal Dual Proximal Operator}, and $(X^*,y^*)$ is such that $(0,0) \in T_{\mathcal{L}}(X^*,y^*)$.
\end{lemma}
\begin{proof}
\par The thesis follows from the developments in \cite[Proposition 1]{ROCK_SIAM_J_CONTROL_OPT}. 
\end{proof}
\par In the following lemma, we bound the solution of every PMM sub-problem encountered by Algorithm \ref{Algorithm PMM-IPM}, while establishing bounds for the proximal estimates $\Xi_k$, and $\lambda_k$.
\begin{lemma} \label{Lemma-boundedness of optimal solutions for sub-problems}
Given Assumptions \textnormal{\ref{Assumption 1}, \ref{Assumption 2}}, there exists a triple $(X_{r_k}^*,y_{r_k}^*,Z_{r_k}^*)$, satisfying: 
\begin{equation} \label{PMM optimal solution}
\begin{split}
A \bm{X}_{r_k}^* + \mu (y_{r_k}^*-\lambda_k) -b & =   0,\\
-\bm{C} + A^\top  y_{r_k}^* + \bm{Z}_{r_k}^* - \mu (\bm{X}_{r_k}^* - \bm{\Xi}_k)& =  0,\\
\langle X_{r_k}^*, Z_{r_k}^*\rangle & = 0,
\end{split}
\end{equation}
\noindent with $X_{r_k}^*, Z_{r_k}^* \in \mathcal{S}^n_+$, and $\|(\bm{X}_{r_k}^*,y_{r_k}^*,\bm{Z}_{r_k}^*)\|_2 = O(\sqrt{n})$, for all $\lambda_k \in \mathbb{R}^m$, $\Xi_k \in \mathcal{S}_+^n$, produced by Algorithm \textnormal{\ref{Algorithm PMM-IPM}}, and any $\mu \in [0,\infty)$. Moreover, $\|(\bm{\Xi}_k,\lambda_k)\|_2 = O(\sqrt{n})$, for all $k \geq 0$.
\end{lemma}
\begin{proof}
\par We prove the claim by induction on the iterates, $k \geq 0$, of Algorithm \ref{Algorithm PMM-IPM}. At iteration $k = 0$, we have that $\lambda_0 = y_0$ and $\Xi_0 = X_0$. But from the construction of the starting point in \eqref{starting point}, we know that $\|(X_0,y_0)\|_2 = O(\sqrt{n})$. Hence, $\|(\Xi_0,\lambda_0)\|_2 = O(\sqrt{n})$ (assuming $n > m$). Invoking Lemma \ref{Lemma non-expansiveness}, there exists a unique pair $(X_{r_0}^*,y_{r_0}^*)$ such that:
$$(X_{r_0}^*,y_{r_0}^*) = \mathcal{P}_0(\Xi_0,\lambda_0),\qquad \|(\bm{X}_{r_0}^*,y_{r_0}^*) - (\bm{X}^*,y^*)\|_{2} \leq \|(\bm{\Xi}_0,\lambda_0)-(\bm{X}^*,y^*)\|_{2},$$
\noindent where $(X^*,y^*,Z^*)$ solves \eqref{non-regularized primal}--\eqref{non-regularized dual}, and from Assumption \ref{Assumption 1}, is such that $\|\bm{X}^*,y^*,\bm{Z}^*\|_2 = O(\sqrt{n})$. Using the triangular inequality, and combining the latter inequality with our previous observations, yields that $\|(\bm{X}_{r_0}^*,y_{r_0}^*)\|_2 = O(\sqrt{n})$. From the definition of the operator in \eqref{PMM Operator Subproblem}, we know that: 
\begin{equation*}
 -C + \mathcal{A}^* y_{r_0}^* - \mu (X_{r_0}^* - \Xi_0) \ \in \partial \delta_{\mathcal{S}_+^n}(X_{r_0}^*),\qquad
 \mathcal{A}X_{r_0}^* + \mu (y_{r_0}^*-\lambda_0) - b \ = 0,
\end{equation*}
\noindent where $\partial(\delta_{\mathcal{S}_+^n}(\cdot))$ is the sub-differential of the indicator function defined in \eqref{Indicator function}. Hence, there must exist $-Z_{r_0}^* \in \partial \delta_{\mathcal{S}_+^n}(X_{r_0}^*)$ (and thus, $Z_{r_0}^* \in \mathcal{S}^n_{+}$, $\langle X_{r_0},Z_{r_0} \rangle = 0$), such that:
\[Z_{r_0}^* = C  - \mathcal{A}^* y_{r_0}^* + \mu  (X_{r_0}^* - \Xi_0),\quad \langle X^*_{r_0},Z^*_{r_0}\rangle = 0,\quad \|\bm{Z}_{r_0}^*\|_2 = O(\sqrt{n}),\]
\noindent where $\|\bm{Z}_{r_0}^*\|_2 = O(\sqrt{n})$ follows from Assumption \ref{Assumption 2}, combined with $\|(\bm{X}^*_{r_0},y^*_{r_0})\|_2 = O(\sqrt{n})$.

\par Let us now assume that at some iteration $k$ of Algorithm \ref{Algorithm PMM-IPM}, we have $\|(\bm{\Xi}_k,\lambda_k)\|_2 = O(\sqrt{n})$. There are two cases for the subsequent iterations:
\begin{itemize}
\item[\textbf{1.}] The proximal estimates are updated, that is $(\Xi_{k+1},\lambda_{k+1}) = (X_{k+1},y_{k+1})$, or
\item[\textbf{2.}] the proximal estimates stay the same, i.e. $(\Xi_{k+1},\lambda_{k+1}) = (\Xi_k,\lambda_k)$.
\end{itemize}
\par \textbf{Case 1.} We know by construction that this occurs only if the following is satisfied:
$$\|(r_p,\bm{R}_d)\|_2 \leq K_N \frac{\mu_{k+1}}{\mu_0},$$
\noindent where $r_p,\ R_d$ are defined in Algorithm \ref{Algorithm PMM-IPM}. However, from the neighbourhood conditions in \eqref{Small neighbourhood}, we know that:
$$\|\big(r_p + \mu_{k+1}(y_{k+1}-\lambda_k), \bm{R}_d + \mu_{k+1}(\bm{X}_{k+1}-\bm{\Xi}_k)\big)\|_2 \leq K_N \frac{\mu_{k+1}}{\mu_0}.$$
\noindent Combining the last two inequalities by applying the triangular inequality, and using the inductive hypothesis ($\|(\bm{\Xi}_k,\lambda_k)\|_2 = O(\sqrt{n})$), yields that 
\[ \|(\bm{X}_{k+1},y_{k+1})\|_2 \leq \frac{2K_N}{\mu_0} + \|(\bm{\Xi}_k,\lambda_k)\|_2 = O(\sqrt{n}).\]
\noindent Hence, $\|(\bm{\Xi}_{k+1},\lambda_{k+1})\|_2 = O(\sqrt{n})$. Then, we can invoke Lemma \ref{Lemma non-expansiveness}, with $\lambda = \lambda_{k+1}$, $\Xi = \Xi_{k+1}$ and any $\mu \geq 0$, which gives 
$$\|(\bm{X}_{r_{k+1}}^*,y_{r_{k+1}}^*) - (\bm{X}^*,y^*)\|_{2} \leq \|(\bm{\Xi}_{k+1},\lambda_{k+1})-(\bm{X}^*,y^*)\|_{2}.$$
\noindent A simple manipulation shows that $\|(\bm{X}_{r_{k+1}}^*,y_{r_{k+1}}^*)\|_2 = O(\sqrt{n})$. As before, we use \eqref{PMM Operator Subproblem} alongside Assumption \ref{Assumption 2} to show the existence of $-Z_{r_{k+1}}^* \in \partial \delta_{\mathcal{S}^n_{+}}(X_{r_{k+1}}^*)$, such that the triple $(X_{r_{k+1}}^*,y_{r_{k+1}}^*,Z_{r_{k+1}}^*)$ satisfies \eqref{PMM optimal solution} with $\|\bm{Z}_{r_{k+1}}^*\|_2 = O(\sqrt{n})$.

\par \textbf{Case 2.} In this case, we have $(\Xi_{k+1},\lambda_{k+1}) = (\Xi_k,\lambda_k)$, and hence the inductive hypothesis gives us directly that $\|(\bm{\Xi}_{k+1},\lambda_{k+1})\|_2 = O(\sqrt{n})$. As before, there exists a triple $(X_{r_{k+1}}^*,y_{r_{k+1}}^*,Z_{r_{k+1}}^*)$ satisfying \eqref{PMM optimal solution}, with $\|(\bm{X}_{r_{k+1}}^*,y_{r_{k+1}}^*,\bm{Z}_{r_{k+1}}^*)\|_2 = O(\sqrt{n})$. 
\end{proof}
\par In the next lemma we define and bound a triple solving a particular parametrized non-linear system of equations, which is then used in Lemma \ref{Lemma boundedness of x z} in order to prove boundedness of the iterates of Algorithm \ref{Algorithm PMM-IPM}.
\begin{lemma} \label{Lemma tilde point}
Given Assumptions \textnormal{\ref{Assumption 1}, \ref{Assumption 2}}, and $(\Xi_k,\lambda_k)$, produced at an arbitrary iteration $k \geq 0$ of Algorithm \textnormal{\ref{Algorithm PMM-IPM}}, and any $\mu \in [0,\infty)$, there exists a triple $(\tilde{X},\tilde{y},\tilde{Z})$ which satisfies the following system of equations:
\begin{equation}  \label{tilde point conditions}
\begin{split}
A \bm{\tilde{X}} + \mu \tilde{y} & =   b + \bar{b} + \mu \lambda_k + \tilde{b}_k,\\
 A^\top  \tilde{y} + \bm{\tilde{Z}} - \mu  \bm{\tilde{X}} & =  \bm{C} + \bm{\bar{C}} - \mu \bm{\Xi}_k + \bm{\tilde{C}}_k,\\
\tilde{X}\tilde{Z} & = \theta I_n,
\end{split}
\end{equation}
\noindent for some arbitrary $\theta > 0$ ($\theta = \Theta(1)$), with $\tilde{X},\ \tilde{Z} \in \mathcal{S}^n_{++}$ and $\|(\bm{\tilde{X}},\tilde{y},\bm{\tilde{Z}})\|_2 = O(\sqrt{n})$, where $\tilde{b}_{k},\ \tilde{C}_{k}$ are defined in \eqref{Small neighbourhood}, while $\bar{b},\ \bar{C}$ are defined with the starting point in \eqref{starting point}. Furthermore, $\nu_{\min}(\tilde{X}) \geq \xi$ and $\nu_{\min}(\tilde{Z}) \geq \xi$, for some positive $\xi = \Theta(1)$.
\end{lemma}
\begin{proof}
\par Let $k \geq 0$ denote an arbitrary iteration of Algorithm \ref{Algorithm PMM-IPM}. Let also $\bar{b},\ \bar{C}$ as defined in \eqref{starting point}, and $\tilde{b}_{k},\ \tilde{C}_{k}$, as defined in the neighbourhood conditions in \eqref{Small neighbourhood}. Given an arbitrary positive constant $\theta > 0$, we consider the following barrier primal-dual pair:
\begin{equation} \label{tilde non-regularized primal} 
\underset{X \in \mathcal{S}^n}{\text{min}} \ \big( \langle C+\bar{C} + \tilde{C}_k),X\rangle  -\theta \ln (\det (X)) \big), \ \ \text{s.t.}  \  \mathcal{A} X= b + \bar{b} + \tilde{b}_k,   
\end{equation}
\begin{equation} \label{tilde non-regularized dual} 
\underset{y \in \mathbb{R}^m,Z \in \mathcal{S}^n}{\text{max}}  \ \big((b + \bar{b} + \tilde{b}_k)^\top y +\theta \ln(\det(Z))\big), \ \ \text{s.t.}\   \mathcal{A}^*y + Z = C+\bar{C} + \tilde{C}_k.
\end{equation}
\par Let us now define the following triple:
\begin{equation*}
(\hat{X},\hat{y},\hat{Z}) \coloneqq \arg \min_{(X,y,Z)} \big\{\|(\bm{X},\bm{Z})\|_2: A\bm{X} = \tilde{b}_k,\ A^\top  y + \bm{Z} = \tilde{C}_k \}. 
\end{equation*}
\noindent From the neighbourhood conditions \eqref{Small neighbourhood}, we know that $\|(\tilde{b}_k,\bm{\tilde{C}}_k)\|_{\mathcal{S}} \leq \gamma_{\mathcal{S}}\rho$, and from the definition of the semi-norm in \eqref{semi-norm definition}, we have that $\|(\bm{\hat{X}},\bm{\hat{Z}})\|_2 \leq \gamma_{\mathcal{S}} \rho$. Using \eqref{semi-norm definition} alongside Assumption \ref{Assumption 2}, we can also show that $\|\hat{y}\|_2 = \Theta(\|(\bm{\hat{X}},\bm{\hat{Z}})\|_2)$. On the other hand, from the definition of the starting point, we have that $(X_0,Z_0) = \rho(I_n,I_n)$. By defining the following auxiliary point:
$$(\bar{X},\bar{y},\bar{Z}) = (X_0,y_0,Z_0) + (\hat{X},\hat{y},\hat{Z}),$$
\noindent we have that $(1 + \gamma_{\mathcal{S}})\rho(I_n,I_n) \succeq (\bar{X},\bar{Z}) \succeq (1-\gamma_{\mathcal{S}})\rho(I_n,I_n)$, that is, the eigenvalues of these matrices are bounded by constants that are independent of the problem under consideration. By construction, the triple $(\bar{X},\bar{y},\bar{Z})$ is a feasible solution for the primal-dual pair in \eqref{tilde non-regularized primal}--\eqref{tilde non-regularized dual}, giving bounded primal and dual objective values, respectively. This, alongside Weierstrass's theorem on a potential function, can be used to show that the solution of problem \eqref{tilde non-regularized primal}--\eqref{tilde non-regularized dual} is bounded. In other words, for any choice of $\theta > 0$, there must exist a bounded  triple $(X_s^*,y_s^*,Z_s^*)$ solving \eqref{tilde non-regularized primal}--\eqref{tilde non-regularized dual}, i.e.:
\begin{equation*}
\begin{split}
A\bm{X}_s^* = b + \bar{b} + \tilde{b}_k,\quad A^\top  y_s^* + \bm{Z}_s^*  = \bm{C} + \bm{\bar{C}} + \bm{\tilde{C}}_k,\quad
X_s^* Z_s^* = \theta I_n,
\end{split}
\end{equation*}
\noindent such that $\nu_{\max}(X_{s^*}) \leq K_{s^*}$ and $\nu_{\max}(Z_{s^*}) \leq K_{s^*}$, where $K_{s^*} > 0$ is a positive constant. In turn, combining this with Assumption \ref{Assumption 2}  implies that $\|(\bm{X}_s^*,y_s^*,\bm{Z}_s^*)\|_2  = O(\sqrt{n})$.

\par Let us now apply the PMM to \eqref{tilde non-regularized primal}--\eqref{tilde non-regularized dual}, given the estimates $\Xi_k,\ \lambda_k$. We should note at this point that the proximal operator used here is different from that in \eqref{Primal Dual Proximal Operator}, since it is based on a different maximal monotone operator to that in \eqref{Primal Dual Maximal Monotone Operator}. In particular, we associate a single-valued maximal monotone operator to \eqref{tilde non-regularized primal}--\eqref{tilde non-regularized dual}, with graph:
\begin{equation*}
\tilde{T}_{\mathcal{L}}(X,y) \coloneqq \big\{(V,u): V =  (C + \bar{C} + \tilde{C}_k) - \mathcal{A}^*y - \theta X^{-1}, u = \mathcal{A}X-(b+\bar{b}+\tilde{b}_k) \big\}.
\end{equation*}
\noindent As before, the proximal operator is defined as $\tilde{\mathcal{P}} \coloneqq (I_{n+m}+ \tilde{T}_{\mathcal{L}})^{-1}$, and is single-valued and non-expansive.  We let any $\mu \in [0,\infty)$ and define the following penalty function:
\begin{equation*} 
\begin{split}
\tilde{\mathcal{L}}_{\mu,\theta}(X;\Xi_k,\lambda_k)  \coloneqq \ & \langle C + \bar{C} + \tilde{C}_k, X\rangle  + 
 \frac{1}{2}\mu \|X-\Xi_k\|_{F}^2 + \frac{1}{2\mu}\|\mathcal{A}X-(b+\bar{b}+\tilde{b}_k)\|_{2}^2 \\ &  - (\lambda_k)^\top (\mathcal{A}X - (b+\bar{b}+\tilde{b}_k))-\theta \ln(\det(X)).
\end{split}
\end{equation*}
\noindent By defining the variables $y = \lambda_k - \frac{1}{\mu}(\mathcal{A}X - (b+\bar{b}+\tilde{b}_k))$ and $Z = \theta X^{-1}$, we can see that the optimality conditions of this PMM sub-problem are exactly those stated in \eqref{tilde point conditions}. Equivalently, we can find a pair $(\tilde{X},\tilde{y})$ such that $(\tilde{X},\tilde{y}) = \tilde{\mathcal{P}}(\Xi_k,\lambda_k)$ and set $\tilde{Z} = \theta \tilde{X}^{-1}$. We can now use the non-expansiveness of $\tilde{\mathcal{P}}$, as in Lemma \ref{Lemma non-expansiveness}, to obtain:
$$\|(\bm{\tilde{X}},\tilde{y})-(\bm{X}_s^*,y_s^*)\|_{2} \leq \|(\Xi_k,\lambda_k)-(\bm{X}_s^*,y_s^*)\|_{2}.$$
\noindent But we know, from Lemma \ref{Lemma-boundedness of optimal solutions for sub-problems}, that $\|(\bm{\Xi}_k,\lambda_k)\|_2 = O(\sqrt{n})$, $\forall\ k \geq 0$. Combining this with our previous observations, yields that $\|(\bm{\tilde{X}},\tilde{y})\|_2 = O(\sqrt{n})$. Setting $\tilde{Z} = \theta\tilde{X}^{-1}$, gives a triple $(\tilde{X},\tilde{y},\tilde{Z})$ that satisfies \eqref{tilde point conditions}, while $\|(\bm{\tilde{X}},\tilde{y},\bm{\tilde{Z}})\|_2 = O(\sqrt{n})$ (from dual feasibility).  
\par To conclude the proof, let us notice that the value of $\tilde{\mathcal{L}}_{\mu,\theta}(X;\Xi_k,\lambda_k)$ will grow unbounded as $\nu_{\min}(X) \rightarrow 0$ or $\nu_{\max}(X) \rightarrow \infty$. Hence, there must exist a constant $\tilde{K} > 0$, such that the minimizer of this function satisfies $\frac{1}{\tilde{K}} \leq \nu_{\min}(\tilde{X}) \leq \nu_{\max}(\tilde{X}) \leq \tilde{K}$. The relation $\tilde{X}\tilde{Z} = \theta I_n$ then implies that $\frac{\theta}{\tilde{K}} \leq \nu_{\min}(\tilde{Z}) \leq \nu_{\max}(\tilde{Z}) \leq \theta \tilde{K}$. Hence, there exists some $\xi = \Theta(1)$ such that $\nu_{\min}(\tilde{X}) \geq \xi$ and $\nu_{\min}(\tilde{Z}) \geq \xi$. 
\end{proof}
\noindent In the following lemma, we derive boundedness of the iterates of Algorithm \ref{Algorithm PMM-IPM}.

\begin{lemma} \label{Lemma boundedness of x z}
Given Assumptions \textnormal{\ref{Assumption 1}} and \textnormal{\ref{Assumption 2}}, the iterates $(X_k,y_k,Z_k)$ produced by Algorithm \textnormal{\ref{Algorithm PMM-IPM}}, for all $k \geq 0$, are such that:
$$\textnormal{Tr}(X_k) = O(n),\qquad \textnormal{Tr}(Z_k) = O(n),\qquad \|(\bm{X}_k,y_k,\bm{Z}_k)\|_2 = O(n).$$
\end{lemma}
\begin{proof}
\par Let an iterate $(X_k,y_k,Z_k) \in \mathscr{N}_{\mu_k}(\Xi_k,\lambda_k)$, produced by Algorithm \ref{Algorithm PMM-IPM} during an arbitrary iteration $k \geq 0$, be given. Firstly, we invoke Lemma \ref{Lemma tilde point}, from which we have a triple $(\tilde{X},\tilde{y},\tilde{Z})$ satisfying \eqref{tilde point conditions}, for $\mu = \mu_k$. Similarly, by invoking Lemma \ref{Lemma-boundedness of optimal solutions for sub-problems}, we know that there exists a triple $(X_{r_k}^*,y_{r_k}^*,Z_{r_k}^*)$ satisfying \eqref{PMM optimal solution}, with $\mu = \mu_k$. Consider the following auxiliary point:
\begin{equation} \label{auxiliary triple 1}
\bigg((1-\frac{\mu_k}{\mu_0})X_{r_k}^* + \frac{\mu_k}{\mu_0} \tilde{X} - X_k,\ (1-\frac{\mu_k}{\mu_0})y_{r_k}^* +\frac{\mu_k}{\mu_0} \tilde{y} - y_k,\ (1-\frac{\mu_k}{\mu_0})Z_{r_k}^* + \frac{\mu_k}{\mu_0} \tilde{Z} - Z_k\bigg).
\end{equation}
\noindent  Using (\ref{auxiliary triple 1}) and \eqref{PMM optimal solution}-\eqref{tilde point conditions} (for $\mu = \mu_k$), one can observe that:
\begin{equation*}
\begin{split}
A\big((1-\frac{\mu_k}{\mu_0})\bm{X}_{r_k}^* + \frac{\mu_k}{\mu_0} \bm{\tilde{X}} - \bm{X}_k\big) + \mu_k  \big((1-\frac{\mu_k}{\mu_0})y_{r_k}^* + \frac{\mu_k}{\mu_0} \tilde{y} - y_k\big) = \\
(1-\frac{\mu_k}{\mu_0})(A\bm{X}_{r_k}^* + \mu_k  y_{r_k}^*) + \frac{\mu_k}{\mu_0} (A\bm{\tilde{X}}+ \mu_k \tilde{y}) - A\bm{X}_k -\mu_k y_k =\\
(1-\frac{\mu_k}{\mu_0}) (b + \mu_k  \lambda_k) + \frac{\mu_k}{\mu_0} (b + \mu_k \lambda_k + \tilde{b}_k +\bar{b}) - A\bm{X}_k - \mu_k y_k =\\
b +\mu_k \lambda_k + \frac{\mu_k}{\mu_0}(\tilde{b}_k+\bar{b}) - A\bm{X}_k - \mu_k y_k = &\ 0,
\end{split}
\end{equation*}
\noindent where the last equality follows from the definition of the neighbourhood $\mathscr{N}_{\mu_k}(\Xi_k,\lambda_k)$. Similarly, one can show that:
\begin{equation*}
-\mu_k \big((1-\frac{\mu_k}{\mu_0})\bm{X}_{r_k}^* + \frac{\mu_k}{\mu_0} \bm{\tilde{X}} - \bm{X}_k\big) + A^\top \big((1-\frac{\mu_k}{\mu_0})y_{r_k}^* +\frac{\mu_k}{\mu_0} \tilde{y} - y_k\big) + \big((1-\frac{\mu_k}{\mu_0})\bm{Z}_{r_k}^* + \frac{\mu_k}{\mu_0} \bm{\tilde{Z}} - \bm{Z}_k\big) = 0.
\end{equation*}
\noindent By combining the previous two relations, we have:
\begin{equation} \label{Lemma boundedness of x,z, relation 1}
\begin{split}
\big((1-\frac{\mu_k}{\mu_0})\bm{X}_{r_k}^* + \frac{\mu_k}{\mu_0} \bm{\tilde{X}} - \bm{X}_k\big)^\top \big((1-\frac{\mu_k}{\mu_0})\bm{Z}_{r_k}^* + \frac{\mu_k}{\mu_0} \bm{\tilde{Z}} - \bm{Z}_k\big) = &\\
\mu_k\big((1-\frac{\mu_k}{\mu_0})\bm{X}_{r_k}^* + \frac{\mu_k}{\mu_0}\bm{\tilde{X}} - \bm{X}_k\big)^\top \big((1-\frac{\mu_k}{\mu_0})\bm{X}_{r_k}^* + \frac{\mu_k}{\mu_0} \bm{\tilde{X}} - \bm{X}_k\big)\  +\\ \mu_k \big((1-\frac{\mu_k}{\mu_0})y_{r_k}^* + \frac{\mu_k}{\mu_0} \tilde{y} - y_k\big)^\top  \big((1-\frac{\mu_k}{\mu_0})y_{r_k}^* + \frac{\mu_k}{\mu_0} \tilde{y} - y_k\big) \geq &\ 0.
\end{split}
\end{equation}
\noindent Observe that (\ref{Lemma boundedness of x,z, relation 1}) can equivalently be written as:
\begin{equation*}
\begin{split}
\big\langle(1-\frac{\mu_k}{\mu_0})X_{r_k}^* + \frac{\mu_k}{\mu_0} \tilde{X}, Z_k \big\rangle + \big\langle (1-\frac{\mu_k}{\mu_0})Z_{r_k}^* + \frac{\mu_k}{\mu_0}\tilde{Z}, X_k \big\rangle \leq \\
\big\langle(1-\frac{\mu_k}{\mu_0})X_{r_k}^* + \frac{\mu_k}{\mu_0} \tilde{X},(1-\frac{\mu_k}{\mu_0})Z_{r_k}^* + \frac{\mu_k}{\mu_0} \tilde{Z}\big\rangle + \langle X_k, Z_k\rangle.
\end{split}
\end{equation*}
\noindent However, from Lemmas \ref{Lemma-boundedness of optimal solutions for sub-problems} and \ref{Lemma tilde point}, we have that $\tilde{X} \succeq \xi I_n$ and $\tilde{Z}  \succeq \xi I_n$, for some positive constant $\xi = \Theta(1)$, $\langle X_{r_k}^*,Z_k \rangle  \geq 0$, $\langle Z_{r_k}^*, X_k\rangle \geq 0$, while $\|(X_{r_k}^*,Z_{r_k}^*)\|_F = O(\sqrt{n})$, and $\|(\tilde{X},\tilde{Z})\|_F = O(\sqrt{n})$. Furthermore, by definition we have that $ n \mu_k = \langle X_k,Z_k \rangle$. By combining all the previous, we obtain:
\begin{equation} \label{Lemma boundedness of x,z, relation 2}
\begin{split}
\frac{\mu_k}{\mu_0} \xi \big(\textnormal{Tr}(X_k) + \textnormal{Tr}(Z_k) \big) = \\
\frac{\mu_k}{\mu_0} \xi\big(\langle I_n, X_k\rangle + \langle I_n, Z_k\rangle\big) \leq \\
\big\langle(1-\frac{\mu_k}{\mu_0})X_{r_k}^* + \frac{\mu_k}{\mu_0} \tilde{X}, Z_k \big\rangle + \big\langle(1-\frac{\mu_k}{\mu_0})Z_{r_k}^* + \frac{\mu_k}{\mu_0} \tilde{Z}, X_k\big\rangle \leq \\
\big\langle(1-\frac{\mu_k}{\mu_0})X_{r_k}^* + \frac{\mu_k}{\mu_0} \tilde{X},(1-\frac{\mu_k}{\mu_0})Z_{r_k}^* + \frac{\mu_k}{\mu_0} \tilde{Z}\big\rangle + \langle X_k, Z_k \rangle = \\
\frac{\mu_k}{\mu_0}(1-\frac{\mu_k}{\mu_0}) \langle X_{r_k}^*, \tilde{Z}\rangle + \frac{\mu_k}{\mu_0} (1-\frac{\mu_k}{\mu_0}) \langle\tilde{X}, Z_r^*\rangle + (\frac{\mu_k}{\mu_0})^2 \langle\tilde{X}, \tilde{Z} \rangle + \langle X_k, Z_k \rangle = &\ O(n \mu_k ),
\end{split}
\end{equation}
\noindent where the first inequality follows since $X_{r_k}^*,\ Z_{r_k}^*,\  \tilde{X},\ \tilde{Z} \in \mathcal{S}^n_+$ and $(\tilde{X},\tilde{Z}) \succeq \xi (I_n,I_n)$. In the penultimate equality we used \eqref{PMM optimal solution} (i.e. $\langle X_{r_k}^*,Z_{r_k}^*\rangle = 0$). Hence, (\ref{Lemma boundedness of x,z, relation 2}) implies that:
$$\textnormal{Tr}(X_k) = O(n), \qquad \textnormal{Tr}(Z_k) = O(n).$$
\noindent From positive definiteness we have that $\|(X_k,Z_k)\|_F = O(n)$. Finally, from the neighbourhood conditions we know that:
$$\bm{C} - A^\top  y_k - \bm{Z}_k + \mu_k (\bm{X}_k - \bm{\Xi}_k) + \frac{\mu_k}{\mu_0} (\bm{\tilde{C}}_k + \bm{\bar{C}}) = 0.$$
\noindent All terms above (except for $y_k$) have a 2-norm that is bounded by some quantity that is $O(n)$ (note that $\|(\bm{\bar{C}},\bar{b})\|_2 = O(\sqrt{n})$ using Assumption \ref{Assumption 2} and the definition in \eqref{starting point}). Hence, using again Assumption \ref{Assumption 2} (i.e. $A$ is full rank, with singular values independent of $n$ and $m$) yields that $\|y_k\|_2 = O(n)$, and completes the proof. 
\end{proof}
\par In what follows, we provide Lemmas \ref{Auxiliary Lemma bound on scaled matrices}--\ref{Auxiliary Lemma scaled rhs of third block of Newton system}, which we use to prove boundedness of the Newton direction computed at every iteration of Algorithm \ref{Algorithm PMM-IPM}, in Lemma \ref{Lemma boundedness Dx Dz}.
\begin{lemma} \label{Auxiliary Lemma bound on scaled matrices}
Let $D_k = S_k^{-\frac{1}{2}}F_k = S_k^{\frac{1}{2}}E_k^{-1}$, where $S_k = E_k F_k$, and $E_k,\ F_k$ are defined as in the Newton system in \eqref{inexact vectorized Newton System}. Then, for any $M \in \mathbb{R}^{n\times n}$, 
\begin{equation*}
\|D_k^{-T} \bm{M}\|^2_2 \leq \frac{1}{(1-\gamma_{\mu})\mu_k}\|Z_k^{\frac{1}{2}}M Z_k^{\frac{1}{2}}\|_F^2,\quad \|D_k\bm{M}\|_2^2 \leq \frac{1}{(1-\gamma_{\mu})\mu_k} \|X_k^{\frac{1}{2}}M X_k^{\frac{1}{2}}\|_F^2,
\end{equation*}
\noindent where $\gamma_{\mu}$ is defined in \eqref{Small neighbourhood}. Moreover, we have that:
\begin{equation*}
\|D_k^{-T}\|_2^2 \leq \frac{1}{(1-\gamma_{\mu})\mu_k} \|Z_k\|_F^2 = O\bigg(\frac{n^2}{\mu_k}\bigg),\qquad \|D_k\|_2^2 \leq \frac{1}{(1-\gamma_{\mu})\mu_k} \|X_k \|_F^2 = O\bigg( \frac{n^2}{\mu_k}\bigg).
\end{equation*}
\end{lemma}
\begin{proof}
The proof of the first two inequalities follows exactly the developments in \cite[Lemma 5]{ZhouToh_MATH_PROG}. The bound on the 2-norm of the matrix $D_k^{-T}$ follows by choosing $M$ such that $\bm{M}$ is a unit eigenvector, corresponding to the largest eigenvalue of $D_k^{-T}$. Then, $\|D_k^{-T}\bm{M}\|_2^2 = \|D_k^{-T}\|_2^2$. But, we have that:
\begin{equation*}
\begin{split}
 \|D_k^{-T}\bm{M}\|_2^2 \leq \ &\frac{1}{(1-\gamma_{\mu})\mu_k} \|Z_k^{\frac{1}{2}}MZ_k^{\frac{1}{2}}\|_F^2 \\
=\ &  \frac{1}{(1-\gamma_{\mu})\mu_k}\textnormal{Tr}(Z_k M^\top  Z_k M)\\
\leq\ & \frac{1}{(1-\gamma_{\mu})\mu_k} \|Z_k\|_F^2 = O\bigg( \frac{n^2}{\mu_k}\bigg)
 \end{split}
\end{equation*}
\noindent where we used the cyclic property of the trace as well as Lemma \ref{Lemma boundedness of x z}. The same reasoning applies to deriving the bound for $\|D_k\|_2^2$. 
\end{proof}
\begin{lemma} \label{Auxiliary Lemma scaled third block of Newton system}
Let $D_k$ and $S_k$ be defined as in Lemma \textnormal{\ref{Auxiliary Lemma bound on scaled matrices}}. Then, we have that:
\begin{equation*}
\|D_k^{-T} \bm{\Delta X}_k\|_2^2 + \| D_k \bm{\Delta Z}_k\|_2^2 + 2\langle \Delta X_k, \Delta Z_k \rangle = \|S_k^{-\frac{1}{2}} \bm{R}_{\mu,k}\|_2^2,
\end{equation*}
\noindent where $R_{\mu,k} = \sigma_k \mu_k I_n - Z_k^{\frac{1}{2}} X_k Z_k^{\frac{1}{2}}$. Furthermore, 
\begin{equation*}
\|H_{P_k}(\Delta X_k \Delta Z_k) \|_F \leq \frac{\sqrt{\frac{1+\gamma_{\mu}}{1-\gamma_{\mu}}}}{2}\big(\|D_k^{-T}\bm{\Delta X}_k\|^2_2 + \|D_k \bm{\Delta Z}_k\|_2^2 \big),
\end{equation*}
\noindent where $\gamma_{\mu}$ is defined in \eqref{Small neighbourhood}.
\end{lemma}
\begin{proof}
\noindent The equality follows directly by pre-multiplying by $S^{-\frac{1}{2}}$ on both sides of the third block equation of the Newton system in \eqref{inexact vectorized Newton System} and by then taking the 2-norm (see \cite[Lemma 3.1]{Zhang_SIAM_J_OPT}). For a proof of the inequality, we refer the reader to \cite[Lemma 3.3]{Zhang_SIAM_J_OPT}. 
\end{proof}
\begin{lemma} \label{Auxiliary Lemma scaled rhs of third block of Newton system}
Let $S_k$ as defined in Lemma \textnormal{\ref{Auxiliary Lemma bound on scaled matrices}}, and $R_{\mu,k}$ as defined in Lemma \textnormal{\ref{Auxiliary Lemma scaled third block of Newton system}}. Then,
\begin{equation*}
\|S_k^{-\frac{1}{2}} \bm{R}_{\mu,k}\|_2^2 = O(n \mu_k).
\end{equation*}
\end{lemma}
\begin{proof}
The proof is omitted since it follows exactly the developments in  \cite[Lemma 7]{ZhouToh_MATH_PROG}. 
\end{proof}
\par We are now ready to derive bounds for the Newton direction computed at every iteration of Algorithm \ref{Algorithm PMM-IPM}.
\begin{lemma} \label{Lemma boundedness Dx Dz}
Given Assumptions \textnormal{\ref{Assumption 1}} and \textnormal{\ref{Assumption 2}}, and the Newton direction $(\Delta X_k, \Delta y_k, \Delta Z_k)$ obtained by solving system \textnormal{\eqref{inexact vectorized Newton System}} during an arbitrary iteration $k \geq 0$ of Algorithm \textnormal{\ref{Algorithm PMM-IPM}}, we have that:
$$\|H_{P_k}(\Delta X_k \Delta Z_k)\|_F = O(n^{4}\mu),\qquad \|(\bm{\Delta X}_k,\Delta y_k,\bm{\Delta Z}_k)\|_2 = O(n^{3}).$$
\end{lemma}

\begin{proof}
\par Consider an arbitrary iteration $k$ of Algorithm \ref{Algorithm PMM-IPM}. We invoke Lemmas \ref{Lemma-boundedness of optimal solutions for sub-problems}, \ref{Lemma tilde point}, for $\mu = \sigma_k \mu_k$. That is, there exists a triple $(X_{r_k}^*,y_{r_k}^*,Z_{r_k}^*)$ satisfying \eqref{PMM optimal solution}, and a triple $(\tilde{X},\tilde{y},\tilde{Z})$ satisfying \eqref{tilde point conditions}, for $\mu = \sigma_k \mu_k$. Using the centering parameter $\sigma_k$, define: 
\begin{equation} \label{Boundedness dx dz c hat b hat equation}
\begin{split}
 \bm{\hat{C}} =& -\bigg(\frac{\sigma_k}{\mu_0} \bm{\bar{C}} - (1-\sigma_k)\big(\bm{X}_k - \bm{\Xi}_k + \frac{\mu_k}{\mu_0}(\bm{\tilde{X}}-\bm{X}_{r_k}^*)\big)+\frac{1}{\mu_k}\bm{\mathsf{E}}_{d,k}\bigg),\\
  \hat{b} =&  -\bigg(\frac{\sigma_k}{\mu_0} \bar{b} + (1-\sigma_k)\big(y_k - \lambda_k +\frac{\mu_k}{\mu_0}(\tilde{y}-y_{r_k}^*) \big)+ \frac{1}{\mu_k}\epsilon_{p,k}\bigg),
 \end{split}
\end{equation}
\noindent where $\bar{b},\ \bar{C},\ \mu_0$ are given by \eqref{starting point} and $\epsilon_{p,k}$, $\mathsf{E}_{d,k}$ model the errors which occur when system \eqref{exact non-vectorized Newton System} is solved inexactly. Notice that these errors are required to satisfy \eqref{Krylov method termination conditions} at every iteration $k$.  Using Lemmas \ref{Lemma-boundedness of optimal solutions for sub-problems}, \ref{Lemma tilde point}, \ref{Lemma boundedness of x z}, relation \eqref{Krylov method termination conditions}, and Assumption \ref{Assumption 2}, we know that $\|(\bm{\hat{C}},\hat{b})\|_2 = O(n)$. Then, by applying again Assumption \ref{Assumption 2}, we know that there must exist a matrix $\hat{X} \in \mathbb{R}^{n\times n}$ such that $\mathcal{A}\hat{X} = \hat{b},\ \|\hat{X}\|_F = O(n)$, and by setting $\hat{Z} =  \hat{C}  + \mu \hat{X}$, we have that $\|\hat{Z}\|_F = O(n)$ and:
\begin{equation} \label{Boundedness of Dx,Dz, hat point}
 \mathcal{A}\hat{X} = \hat{b},\qquad  \hat{Z} - \mu_k \hat{X} = \hat{C}.
 \end{equation}
\par Using $(X_{r_k}^*,y_{r_k}^*,Z_{r_k}^*)$, $(\tilde{X},\tilde{y},\tilde{Z})$, as well as the triple $(\hat{X},0,\hat{Z})$, where $(\hat{X},\hat{Z})$ is defined in \eqref{Boundedness of Dx,Dz, hat point}, we define the following auxiliary triple:
\begin{equation} \label{Lemma Dx Dz boundedness, auxiliary triple}
(\bar{X},\bar{y},\bar{Z}) = (\Delta X_k, \Delta y_k, \Delta Z_k) + \frac{\mu_k}{\mu_0} (\tilde{X}, \tilde{y}, \tilde{Z}) - \frac{\mu_k}{\mu_0} (X_{r_k}^*, y_{r_k}^*,Z_{r_k}^*) + \mu_k (\hat{X},0,\hat{Z}).
\end{equation}
\noindent Using \eqref{Lemma Dx Dz boundedness, auxiliary triple}, \eqref{Boundedness dx dz c hat b hat equation}, and the second block equation of \eqref{inexact vectorized Newton System}:
\begin{equation*}
\begin{split}
A\bm{\bar{X}} + \mu_k \bar{y} = &\ (A \bm{\Delta X}_k + \mu_k \Delta y_k) + \frac{\mu_k}{\mu_0}((A\bm{\tilde{X}}+  \mu_k \tilde{y})- (A\bm{X}_{r_k}^*+  \mu_k  y_{r_k}^*)) + \mu_k A\bm{\hat{X}}\\
= &\ \big(b + \sigma_k\frac{\mu_k}{\mu_0}\bar{b}-A\bm{X}_k - \sigma_k \mu_k (y_k-\lambda_k) + \epsilon_{p,k}\big) \\ &\ + \frac{\mu_k}{\mu_0}((A\bm{\tilde{X}} +  \mu_k \tilde{y})- (A\bm{X}_{r_k}^*+  \mu_k y_{r_k}^*))\\
 &\ - \mu_k \big(\sigma_k \frac{\bar{b}}{\mu_0} + (1-\sigma_k)(y_k-\lambda_k)\big) - \frac{\mu_k}{\mu_0}(1-\sigma_k)\mu_k (\tilde{y}-y_{r_k}^*) - \epsilon_{p,k}.
\end{split}
\end{equation*}

\noindent Then, by deleting opposite terms in the right-hand side, and employing \eqref{PMM optimal solution}-\eqref{tilde point conditions} (evaluated at $\mu = \sigma_k \mu_k$ from the definition of  $(X_{r_k}^*,y_{r_k}^*,Z_{r_k}^*)$ and  $(\tilde{X},\tilde{y},\tilde{Z})$), we have
\begin{equation*}
\begin{split}
A\bm{\bar{X}} + \mu_k \bar{y} = &\ \big(b + \sigma_k\frac{\mu_k}{\mu_0}\bar{b}-A\bm{X}_k - \sigma_k \mu_k (y_k-\lambda_k)\big) + \frac{\mu_k}{\mu_0}(b+\sigma_k\mu_k  \lambda_k+\bar{b}+\tilde{b}_k)\\
 &\ - \frac{\mu_k}{\mu_0} (\sigma_k \mu_k \lambda_k + b) - \mu_k \big(\sigma_k \frac{\bar{b}}{\mu_0} + (1-\sigma_k)(y_k-\lambda_k)\big)\\
 = &\ b + \frac{\mu_k}{\mu_0}(\bar{b}+\tilde{b}_k) - A\bm{X}_k - \mu_k  (y_k-\lambda_k)\\
 = &\ 0, 
\end{split}
\end{equation*}
\noindent where the last equation follows from the neighbourhood conditions (i.e. $(X_k,y_k,Z_k) \in \mathscr{N}_{\mu_k}(\Xi_k,\lambda_k)$). Similarly, we can show that:
$$ A^\top \bar{y} + \bar{Z}-\mu_k \bar{X} = 0.$$
\par The previous two equalities imply that: 
\begin{equation} \label{Lemma boundedness Dx Dz, complementarity positivity}
\begin{split}
\langle \bar{X},\bar{Z}\rangle = 
\langle \bar{X}, - \mathcal{A^*} \bar{y}+\mu_k \bar{X}\rangle = \mu_k \langle \bar{X}, \bar{X} \rangle + \mu_k \bar{y}^\top \bar{y} \geq 0.
\end{split}
\end{equation}

\noindent On the other hand, using the last block equation of the Newton system (\ref{inexact vectorized Newton System}), we have:
\begin{equation*}
E_k\bm{\bar{X}} + F_k \bm{\bar{Z}} = \bm{R_{\mu,k}}+ \frac{\mu_k}{\mu_0} E_k(\bm{\tilde{X}}-\bm{X}_{r_k}^* + \mu_0 \bm{\hat{X}})+\frac{\mu_k}{\mu_0}  F_k(\bm{\tilde{Z}}- \bm{Z}_{r_k}^* + \mu_0 \bm{\hat{Z}}),
\end{equation*}
\noindent where $R_{\mu,k}$ is defined as in Lemma \ref{Auxiliary Lemma scaled third block of Newton system}. Let $S_k$ be defined as in Lemma \ref{Auxiliary Lemma bound on scaled matrices}. By multiplying both sides of the previous equation by $S_k^{-\frac{1}{2}}$, we get:
\begin{equation} \label{Lemma boundedness Dx Dz, relation 1}
D_k^{-T}\bm{\bar{X}} + D_k\bm{\bar{Z}} = S_k^{-\frac{1}{2}}\bm{R_{\mu,k}}+ \frac{\mu_k}{\mu_0} \big(D_k^{-T}(\bm{\tilde{X}}-\bm{X}_{r_k}^* + \mu_0 \bm{\hat{X}}) + D_k(\bm{\tilde{Z}}-\bm{Z}_{r_k}^* + \mu_0 \bm{\hat{Z}})\big).
\end{equation}
\noindent But from (\ref{Lemma boundedness Dx Dz, complementarity positivity}) we know that $\langle\bar{X}, \bar{Z}\rangle \geq 0$ and hence:
\begin{equation*}
\|D_k^{-T}\bm{\bar{X}} + D_k \bm{\bar{Z}}\|_2^2 \geq \|D_k^{-T} \bm{\bar{X}}\|_2^2 + \|D_k \bm{\bar{Z}}\|_2^2.
\end{equation*}
\noindent Combining (\ref{Lemma boundedness Dx Dz, relation 1}) with the previous inequality, gives:
\begin{equation*}
\begin{split}
\|D_k^{-T}\bm{\bar{X}}\|_2^2  \leq \ \bigg\{&\|S_k^{-\frac{1}{2}}\bm{R}_{\mu,k}\|_2 +\frac{\mu_k}{\mu_0} \bigg(\|D_k^{-T}(\bm{\tilde{X}}-\bm{X}_{r_k}^* + \mu_0 \bm{\hat{X}})\|_2  \\ &\ +  \|D_k(\bm{\tilde{Z}}-\bm{Z}_{r_k}^* + \mu_0 \bm{\hat{Z}})\|_2\bigg) \bigg\}^2.
\end{split}
\end{equation*}
\par We take square roots, use \eqref{Lemma Dx Dz boundedness, auxiliary triple} and apply the triangular inequality, to get:
\begin{equation} \label{Lemma boundedness Dx Dz, relation 2}
\begin{split}
\|D_k^{-T} \bm{\Delta X}_k \|_2 \leq &\ \|S_k^{-\frac{1}{2}} \bm{R}_{\mu,k}\|_2
+ \frac{\mu_k}{\mu_0}\bigg( 2\|D_k^{-T} (\bm{\tilde{X}}-\bm{X}_{r_k}^* + \mu_0 \bm{\hat{X}})\|_2 \\ &\ +\|D_k(\bm{\tilde{Z}}-\bm{Z}_{r_k}^* + \mu_0 \bm{\hat{Z}})\|_2\bigg).
\end{split}
\end{equation} 
\par We now proceed to bounding the terms in the right hand side of (\ref{Lemma boundedness Dx Dz, relation 2}). A bound for the first term of the right hand side is given by Lemma \ref{Auxiliary Lemma scaled rhs of third block of Newton system}, that is:
\[ \|S_k^{-\frac{1}{2}} \bm{R}_{\mu,k}\|_2  = O(n^{\frac{1}{2}}\mu_k^{\frac{1}{2}}).\]
\noindent On the other hand, we have (from Lemma \ref{Auxiliary Lemma bound on scaled matrices}) that 
\[\|D_k^{-T}\|_2 = O\Bigg( \frac{n}{\mu_k^{\frac{1}{2}}}\Bigg),\qquad \|D_k\|_2 = O\Bigg( \frac{n}{\mu_k^{\frac{1}{2}}}\Bigg).\]
\noindent Hence, using the previous bounds, as well as Lemmas \ref{Lemma-boundedness of optimal solutions for sub-problems}, \ref{Lemma tilde point}, and \eqref{Boundedness of Dx,Dz, hat point}, we obtain:
\begin{equation*}
\begin{split}
2\frac{\mu_k}{\mu_0}\|D_k^{-T}(\bm{\tilde{X}}-\bm{X}_{r_k}^* + \mu_0 \bm{\hat{X}})\|_2 +\frac{\mu_k}{\mu_0}\|D_k(\bm{\tilde{Z}}-\bm{Z}_{r_k}^* + \mu_0 \bm{\hat{Z}})\|_2 =  O\big(n^{2}\mu_k^{\frac{1}{2}}\big),
\end{split}
\end{equation*}
\noindent Combining all the previous bounds yields that $\|D_k^{-T}\bm{\Delta X}_k\|_2 = O(n^2 \mu_k^{\frac{1}{2}})$. One can bound $\|D_k \bm{\Delta Z}_k\|_2$ in the same way. The latter is omitted for ease of presentation. 
\par Furthermore, we have that:
$$\|\bm{\Delta X}_k\|_2 = \|D_k D_k^{-T} \bm{\Delta X}_k\|_2 \leq \|D_k\|_2\|D_k^{-T} \bm{\Delta X}_k\|_2 = O(n^{3}).$$
\noindent Similarly, we can show that $\|\bm{\Delta Z}_k \|_2 = O(n^{3})$. From the first block equation of the Newton system in \eqref{inexact vectorized Newton System}, alongside Assumption \ref{Assumption 2}, we can show that $\|\Delta y_k\|_2 = O(n^{3})$.
\par Finally, using the previous bounds, as well as Lemma \ref{Auxiliary Lemma scaled third block of Newton system}, we obtain the desired bound on $\|H_{P_k}(\Delta X_k \Delta Z_k)\|_F$, that is:
\[\|H_{P_k}(\Delta X_k \Delta Z_k)\|_F = O(n^4\mu_k),\]
\noindent which completes the proof.  
\end{proof}

\noindent We can now prove (Lemmas \ref{Lemma step-length-part 1}--\ref{Lemma step-length-part 2}) that at every iteration of Algorithm \ref{Algorithm PMM-IPM} there exists a step-length $\alpha_k > 0$, using which, the new iterate satisfies the conditions required by the algorithm. The lower bound on any such step-length will later determine the polynomial complexity of the method. To that end, we assume the following notation:
\[\big(X_k(\alpha),y_k(\alpha),Z_k(\alpha)\big) \equiv (X_k + \alpha \Delta X_k, y_k + \alpha \Delta y_k, Z_k + \alpha \Delta Z_k).\]
\begin{lemma} \label{Lemma step-length-part 1}
Given Assumptions  \textnormal{\ref{Assumption 1}}, \textnormal{\ref{Assumption 2}}, and by letting ${P_k}(\alpha) = Z_k(\alpha)^{\frac{1}{2}}$, there exists a step-length ${\alpha^*} \in (0,1)$, such that for all $\alpha \in [0,{\alpha^*}]$ and for all iterations $k \geq 0$ of Algorithm \textnormal{\ref{Algorithm PMM-IPM}}, the following relations hold:
\begin{equation} \label{Lemma step-length relation 1}
\langle X_k + \alpha \Delta X_k,Z_k + \alpha \Delta Z_k\rangle \geq (1-\alpha(1-\beta_1))\langle X_k,Z_k \rangle,
\end{equation}
\begin{equation} \label{Lemma step-length relation 2}
\|H_{{P_k}(\alpha)}(X_k(\alpha)Z_k(\alpha)) - \mu_k(\alpha)\|_F \leq \gamma_{\mu}\mu_k(\alpha),
\end{equation}
\begin{equation} \label{Lemma step-length relation 3}
\langle X_k + \alpha \Delta X_k, Z_k + \alpha \Delta Z_k \rangle \leq (1-\alpha(1-\beta_2))\langle X_k, Z_k \rangle,
\end{equation}
where, without loss of generality, $\beta_1 = \frac{\sigma_{\min}}{2}$ and $\beta_2 = 0.99$. Moreover, ${\alpha^*} \geq \frac{{\kappa^*}}{n^{4}}$ for all $k\geq 0$, where ${\kappa^*} > 0$ is independent of $n$, $m$.
\end{lemma}
\begin{proof}
\par From Lemma \ref{Lemma boundedness Dx Dz}, there exist constants $K_{\Delta} >0$ and $K_{H\Delta} > 0$, independent of $n$ and $m$, such that:
$$\langle \Delta X_k, \Delta Z_k \rangle = (D_k^{-T} \bm{\Delta X}_k)^\top  (D_k \bm{\Delta Z}_k) \leq \|D_k^{-T} \bm{\Delta X}_k\|_2 \|D_k \bm{\Delta Z}_k\|_2 \leq K_{\Delta}^2 n^4 \mu_k,$$
\[ \|H_{P_k}(\Delta X_k \Delta Z_k)\|_F \leq K_{H\Delta} n^4 \mu_k.\]
\noindent From the last block equation of the Newton system \eqref{exact non-vectorized Newton System}, we can show that:
\begin{equation} \label{Lemma step-length equation 1}
\langle Z_k, \Delta X_k\rangle + \langle X_k, \Delta Z_k\rangle = (\sigma_k - 1) \langle X_k, Z_k \rangle.
\end{equation}
\noindent The latter can also be obtained from \eqref{inexact vectorized Newton System}, since we require $\mathsf{E}_{\mu,k}= 0$. Furthermore:
\begin{equation} \label{Lemma step-length equation 2}
H_{P_k}(X_k(\alpha)Z_k(\alpha)) = (1-\alpha)H_{P_k}(X_k Z_k) + \alpha \sigma_k \mu_k I_n +\alpha^2 H_{P_k}(\Delta X_k \Delta Z_k),
\end{equation}
\noindent where $(X_{k+1},y_{k+1},Z_{k+1}) = (X_k + \alpha\Delta X_k,y_k + \alpha\Delta y_k, Z_k + \alpha\Delta Z_k)$.
\par We proceed by proving \eqref{Lemma step-length relation 1}. Using \eqref{Lemma step-length equation 1}, we have:
\begin{equation*}
\begin{split}
\langle X_k + \alpha \Delta X_k,Z_k + \alpha \Delta Z_k\rangle - (1-\alpha(1 -\beta_1))\langle X_k, Z_k\rangle = \\
\langle X_k, Z_k\rangle +\alpha (\sigma_k - 1)\langle X_k, Z_k \rangle +  \alpha^2 \langle \Delta X_k, \Delta Z_k \rangle - (1-\alpha)\langle X_k, Z_k \rangle -\alpha \beta_1 \langle X_k, Z_k \rangle \geq \\
\alpha (\sigma_k - \beta_1) \langle X_k, Z_k\rangle - \alpha^2 K_{\Delta}^2 n^4 \mu_k \geq \alpha (\frac{\sigma_{\min}}{2})n \mu_k - \alpha^2 K_{\Delta}^2 n^4 \mu_k,
\end{split}
\end{equation*}
\noindent where we set (without loss of generality) $\beta_1 = \frac{\sigma_{\min}}{2}$. The most-right hand side of the previous inequality will be non-negative for every $\alpha$ satisfying:
$$\alpha \leq \frac{\sigma_{\min}}{2 K_{\Delta}^2 n^3}.$$ 
\par In order to prove \eqref{Lemma step-length relation 2}, we will use \eqref{Lemma step-length equation 2} and the fact that from the neighbourhood conditions we have that $\|H_{P_k}(X_k Z_k) - \mu_k\|_F \leq \gamma_{\mu} \mu_k$. For that, we use the result in \cite[Lemma 4.2]{Zhang_SIAM_J_OPT}, stating that:
\[ \|H_{{P_k}(\alpha)}(X_k(\alpha) Z_k(\alpha)) - \mu_k(\alpha) I_n\|_F \leq \|H_{P_k}(X_k(\alpha) Z_k(\alpha)) - \mu_k(\alpha) I_n\|_F. \]
\noindent By combining all the previous, we have:
\begin{equation*}
\begin{split}
\|H_{{P_k}(\alpha)}(X_k(\alpha) Z_k(\alpha)) - \mu_k(\alpha) I_n\|_F - \gamma_{\mu}\mu_k(\alpha) &\ \leq \\
 \|H_{P_k}(X_k(\alpha) Z_k(\alpha)) - \mu_k(\alpha) I_n\|_F - \gamma_{\mu} \mu_k(\alpha) &\ = \\
 \|(1-\alpha)(H_{P_k}(X_k Z_k)-\mu_k I_n)  + \alpha^2 H_{P_k}(\Delta X_k, \Delta Z_k)  - \frac{\alpha^2}{n} \langle \Delta X_k, \Delta Z_k \rangle I_n\|_F - \gamma_{\mu} \mu_k(\alpha) &\ \leq \\
 (1-\alpha)\|H_{P_k}(X_kZ_k) - \mu_k I_n\|_F + \alpha^2\mu_k \bigg(\frac{K_{\Delta}^2}{n} + K_{H\Delta}\bigg)n^4 &\ \\ - \gamma_{\mu} \bigg((1-\alpha)\mu_k + \alpha\sigma_k \mu_k + \frac{\alpha^2}{n}\langle \Delta X_k, \Delta Z_k \rangle \bigg) &\ \leq\\
 -\gamma_{\mu} \alpha \sigma_{\min} \mu_k + \alpha^2 \mu_k \bigg(\frac{2K_{\Delta}^2}{n} + K_{H\Delta} \bigg)n^4,
\end{split}
\end{equation*}
\noindent where we used the neighbourhood conditions in \eqref{Small neighbourhood}, the equality $\mu_k(\alpha) = (1-\alpha)\mu_k + \alpha\sigma_k \mu_k + \frac{\alpha^2}{n} \langle\Delta X_k, \Delta Z_k \rangle$ (which can be derived from \eqref{Lemma step-length equation 1}), and the third block equation of the Newton system \eqref{inexact vectorized Newton System}. The most-right hand side of the previous is non-positive for every $\alpha$ satisfying:
$$\alpha \leq \frac{\sigma_{\min}\gamma_{\mu}}{\big(\frac{2K_{\Delta}^2}{n} + K_{H\Delta}\big)  n^4}.$$
\par Finally, to prove (\ref{Lemma step-length relation 3}), we set (without loss of generality) $\beta_2 = 0.99$. We know, from Algorithm \ref{Algorithm PMM-IPM}, that $\sigma_{\max} \leq 0.5$. With the previous two remarks in mind, we have:
\begin{equation*}
\begin{split}
\frac{1}{n}\langle X_k + \alpha \Delta X_k, Z_k + \alpha \Delta Z_k \rangle - (1-0.01\alpha)\mu_k \leq \\
(1-\alpha)\mu_k + \alpha \sigma_k \mu_k + \alpha^2 \frac{K_{\Delta}^2 n^4}{n}\mu_k - (1-0.01 \alpha)\mu_k \leq \\
-0.99\alpha \mu_k + 0.5\alpha \mu_k + \alpha^2 \frac{K_{\Delta}^2 n^4}{n} \mu_k =\\
-0.49\alpha \mu_k +\alpha^2\frac{K_{\Delta}^2 n^4}{n}\mu_k.
\end{split}
\end{equation*}
\noindent The last term will be non-positive for every $\alpha$ satisfying:
$$\alpha \leq \frac{0.49 }{K_{\Delta}^2 n^3}.$$
\par By combining all the previous bounds on the step-length, we have that \eqref{Lemma step-length relation 1}-\eqref{Lemma step-length relation 3} hold for every $\alpha \in (0,\alpha^*)$, where:
\begin{equation} \label{Lemma step-length bound on step-length}
\alpha^* \coloneqq \min\bigg\{ \frac{\sigma_{\min}}{2 K_{\Delta}^2 n^3},\ \frac{\sigma_{\min}\gamma_{\mu}}{\big(\frac{2K_{\Delta}^2}{n} + K_{H\Delta}\big)  n^4},\ \frac{0.49 }{K_{\Delta}^2 n^3},\ 1\bigg\}.
\end{equation}
\noindent Since ${\alpha^*} = \Omega\big(\frac{1}{n^{4}}\big)$, we know that there must exist a constant ${\kappa^*} > 0$, independent of $n$, $m$ and of the iteration $k$, such that ${\alpha^*} \geq \frac{\kappa}{n^{4}}$, for all $k \geq 0$, and this completes the proof.
\end{proof}
\begin{lemma} \label{Lemma step-length-part 2}
Given Assumptions  \textnormal{\ref{Assumption 1}}, \textnormal{\ref{Assumption 2}}, and by letting ${P_k}(\alpha) = Z_k(\alpha)^{\frac{1}{2}}$, there exists a step-length $\bar{\alpha} \geq \frac{\bar{\kappa}}{n^{4}} \in (0,1)$, where $\bar{\kappa} > 0$ is independent of $n$, $m$, such that for all $\alpha \in [0,\bar{\alpha}]$ and for all iterations $k \geq 0$ of Algorithm \textnormal{\ref{Algorithm PMM-IPM}}, if $(X_k,y_k,Z_k) \in \mathscr{N}_{\mu_k}(\Xi_k,\lambda_k)$, then letting:
$$(X_{k+1},y_{k+1},Z_{k+1}) = (X_k + \alpha\Delta X_k,y_k + \alpha\Delta y_k, Z_k + \alpha\Delta Z_k),\ \mu_{k+1} = \frac{\langle X_{k+1},Z_{k+1}\rangle}{n},$$
\noindent for any $\alpha \in (0,\bar{\alpha}]$, gives $(X_{k+1},y_{k+1},Z_{k+1}) \in \mathscr{N}_{\mu_{k+1}}(\Xi_{k+1},\lambda_{k+1})$, where $\Xi_k,$ and $\lambda_k$ are updated as in Algorithm \textnormal{\ref{Algorithm PMM-IPM}}.
\end{lemma}
\begin{proof}
\noindent Let $\alpha^*$ be given as in Lemma \ref{Lemma step-length-part 1} (i.e. in \eqref{Lemma step-length bound on step-length}) such that \eqref{Lemma step-length relation 1}--\eqref{Lemma step-length relation 3} are satisfied. We would like to find the maximum $\bar{\alpha} \in (0,\alpha^*)$, such that: 
$$(X_k(\alpha),y_k(\alpha),Z_k(\alpha)) \in \mathscr{N}_{\mu_k(\alpha)}(\Xi_k,\lambda_k),\ \text{for all}\ \alpha \in (0,\bar{\alpha}),$$
\noindent where $\mu_k(\alpha) = \frac{\langle X_k(\alpha), Z_k(\alpha) \rangle}{n}$. Let:
\begin{equation} \label{primal infeasibility vector}
\tilde{r}_p(\alpha) = A\bm{X}_k(\alpha)+ \mu_k(\alpha)(y_k(\alpha)-\lambda_k) - \big(b + \frac{\mu_k(\alpha)}{\mu_0}\bar{b}\big),
\end{equation} 
\noindent and
\begin{equation}\label{dual infeasibility matrix}
\bm{\tilde{R}}_d(\alpha) =  A^\top y_k(\alpha) + \bm{Z}_k(\alpha) - \mu_k(\alpha)(\bm{X}_k(\alpha)- \bm{\Xi}_k) - \big(\bm{C} + \frac{\mu_k(\alpha)}{\mu_0}\bm{\bar{C}}\big).
\end{equation}
\noindent In other words, we need to find the maximum $\bar{\alpha} \in (0,\alpha^*)$, such that:
\begin{equation} \label{Step-length neighbourhood conditions}
\|\tilde{r}_p(\alpha),\bm{\tilde{R}}_d(\alpha)\|_2 \leq K_N \frac{\mu_k(\alpha)}{\mu_0},\ \ \|\tilde{r}_p(\alpha),\bm{\tilde{R}}_d(\alpha)\|_{\mathcal{S}} \leq \gamma_{\mathcal{S}} \rho\frac{\mu_k(\alpha)}{\mu_0} ,\ \text{for all}\ \alpha \in (0,\bar{\alpha}).
\end{equation}
\noindent If the latter two conditions hold, then $(X_k(\alpha),y_k(\alpha),Z_k(\alpha)) \in \mathscr{N}_{\mu_k(\alpha)}(\Xi_k,\lambda_k),\ \text{for all}\ \alpha \in (0,\bar{\alpha})$. Then, if Algorithm \ref{Algorithm PMM-IPM} updates $\Xi_k$, and $\lambda_k$, it does so only when similar conditions (as in \eqref{Step-length neighbourhood conditions}) hold for the new parameters. Indeed, notice that the estimates $\Xi_k$ and $\lambda_k$ are only updated if the last conditional of Algorithm \ref{Algorithm PMM-IPM} is satisfied. But this is equivalent to saying that \eqref{Step-length neighbourhood conditions} is satisfied after setting $\bm{\Xi}_k = \bm{X}_k(\alpha)$ and $\lambda_k = y_k(\alpha)$. On the other hand, if the parameters are not updated, the new iterate lies in the desired neighbourhood because of \eqref{Step-length neighbourhood conditions}, alongside \eqref{Lemma step-length relation 1}--\eqref{Lemma step-length relation 3}.
\par We start by rearranging $\tilde{r}_p(\alpha)$. Specifically, we have that:\begin{equation*} 
\begin{split}
\tilde{r}_p(\alpha) &= \\ A(\bm{X}_k + \alpha \bm{\Delta X}_k) +\big(\mu_k + \alpha(\sigma_k-1)\mu_k +\frac{\alpha^2}{n}\langle\Delta X_k,\Delta Z_k\rangle\big)\big((y_k + \alpha \Delta y_k -\lambda_k)-\frac{\bar{b}}{\mu_0} \big) -b &=\\
  \big(A\bm{X}_k +\mu_k (y_k -\lambda_k)-b -\frac{\mu_k}{\mu_0}\bar{b}\big) + \alpha(A \bm{\Delta X}_k + \mu_k \Delta y_k) \\
 +\ \big(\alpha(\sigma_k-1)\mu_k + \frac{\alpha^2}{n}\langle\Delta X_k, \Delta Z_k\rangle \big)\big((y_k - \lambda_k + \alpha \Delta y_k) - \frac{\bar{b}}{\mu_0}\big)&=\\
  \frac{\mu_k}{\mu_0}\tilde{b}_k + \alpha\bigg(b- A\bm{X}_k - \sigma_k\mu_k\big((y_k-\lambda_k)-\frac{\bar{b}}{\mu_0} \big) + \epsilon_{p,k} + \mu_k \big((y_k-\lambda_k)-\frac{\bar{b}}{\mu_0} \big)\ \\
  -\ \mu_k  \big((y_k-\lambda_k)-\frac{\bar{b}}{\mu_0} \big) \bigg) + \big(\alpha(\sigma_k-1)\mu_k + \frac{\alpha^2}{n}\langle \Delta X_k, \Delta Z_k\rangle\big)\big((y_k - \lambda_k + \alpha \Delta y_k) - \frac{\bar{b}}{\mu_0}\big).
\end{split}
\end{equation*}
\noindent where we used the definition of $\tilde{b}_k$ in the neighbourhood conditions in \eqref{Small neighbourhood}, and the second block equation in \eqref{inexact vectorized Newton System}. By using again the neighbourhood conditions, and then by deleting the opposite terms in the previous equation, we obtain:\begin{equation}\label{primal infeasibility formula}
\begin{split}
\tilde{r}_p(\alpha) = &\ (1-\alpha)\frac{\mu_k}{\mu_0}\tilde{b}_k + \alpha \epsilon_{p,k} + \alpha^2(\sigma_k - 1)\mu_k \Delta y_k + \frac{\alpha^2}{n}\langle \Delta X_k, \Delta Z_k\rangle\big(y_k - \lambda_k + \alpha \Delta y_k - \frac{\bar{b}}{\mu_0} \big).
\end{split}
\end{equation}
\noindent Similarly, we can show that:
\begin{equation}\label{dual infeasibility formula}
\bm{\tilde{R}}_d(\alpha) = (1-\alpha)\frac{\mu_k}{\mu_0}\bm{\tilde{C}}_k + \alpha \bm{\mathsf{E}}_{d,k}- \alpha^2(\sigma_k-1)\mu_k \bm{\Delta X}_k - \frac{\alpha^2}{n}\langle \Delta X_k, \Delta Z_k \rangle \big(\bm{X}_k - \bm{\Xi}_k + \alpha \bm{\Delta X}_k + \frac{1}{\mu_0} \bm{\bar{C}}\big).
\end{equation}
\par Recall (Lemma \ref{Lemma boundedness Dx Dz}) that $\langle \Delta X_k, \Delta Z_k \rangle \leq K_{\Delta}^2 n^4 \mu_k$, and define the following quantities
\begin{equation} \label{step-length, auxiliary constants}
\begin{split}
\xi_2 = &\ \mu_k \|(\Delta y_k,\bm{\Delta X}_k)\|_2  +  K_{\Delta}^2n^3  \mu_{k}\bigg(\|(y_k - \lambda_k,\bm{X}_k-\bm{\Xi}_k)\|_2\ +\\
&\  \alpha^* \|(\Delta y_k,\bm{\Delta X}_k)\|_2 + \frac{1}{\mu_0}\|(\bar{b},\bm{\bar{C}})\|_2\bigg),\\
\xi_{\mathcal{S}} = &\ \mu_k \|(\Delta y_k,\bm{\Delta X}_k)\|_{\mathcal{S}} + K_{\Delta}^2n^3 \mu_{k}\bigg(\|(y_k - \lambda_k,\bm{X}_k-\bm{\Xi}_k)\|_{\mathcal{S}} \ +\\
&\  \alpha^* \|(\Delta y_k,\bm{\Delta X}_k)\|_{\mathcal{S}} +  \frac{1}{\mu_0}\|(\bar{b},\bm{\bar{C}})\|_{\mathcal{S}}\bigg),
\end{split}
\end{equation}
\noindent where $\alpha^*$ is given by \eqref{Lemma step-length bound on step-length}. Using the definition of the starting point in \eqref{starting point}, as well as results in Lemmas \ref{Lemma boundedness of x z}, \ref{Lemma boundedness Dx Dz}, we can observe that $\xi_2 = O(n^{4} \mu_k)$. On the other hand, using Assumption \ref{Assumption 2}, we know that for every pair $(r_1,\bm{R}_2) \in \mathbb{R}^{m+n^2}$ (where $R_2 \in \mathbb{R}^{n \times n}$ is an arbitrary matrix), if $\|(r_1,\bm{R}_2)\|_2 = \Theta(f(n))$, where $f(\cdot)$ is a positive polynomial function of $n$, then $\|(r_1,R_2)\|_{\mathcal{S}} = \Theta(f(n))$. Hence, we have that $\xi_{\mathcal{S}} = O(n^{4}\mu_k)$. Using the quantities in \eqref{step-length, auxiliary constants}, equations \eqref{primal infeasibility formula}, \eqref{dual infeasibility formula}, as well as the neighbourhood conditions, we have that:
\begin{equation*}
\begin{split}
\|\tilde{r}_p(\alpha),\bm{\tilde{R}}_d(\alpha)\|_2 \leq &\ (1-\alpha)K_N \frac{\mu_k}{\mu_0} + \alpha \mu_k \|(\epsilon_{p,k},\bm{\mathsf{E}}_{d,k})\|_2+ \alpha^2 \mu_k \xi_2,\\
\|\tilde{r}_p(\alpha),\bm{\tilde{R}}_d(\alpha)\|_S \leq &\  (1-\alpha)\gamma_{\mathcal{S}}\rho \frac{\mu_k}{\mu_0} + + \alpha \mu_k \|(\epsilon_{p,k},\bm{\mathsf{E}}_{d,k})\|_{\mathcal{S}} + \alpha^2 \mu_k \xi_{\mathcal{S}},
\end{split}
\end{equation*}
\noindent for all $\alpha \in (0,\alpha^*)$, where $\alpha^*$ is given by \eqref{Lemma step-length bound on step-length} and the error occurring from the inexact solution of \eqref{exact non-vectorized Newton System}, $(\epsilon_{p,k},\mathsf{E}_{d,k})$, satisfies \eqref{Krylov method termination conditions}. From \eqref{Lemma step-length relation 1}, we know that: 
$$\mu_k(\alpha) \geq (1-\alpha(1-\beta_1))\mu_k,\ \text{for all}\ \alpha \in (0,\alpha^*).$$
\noindent By combining the last three inequalities, using \eqref{Krylov method termination conditions} and setting $\beta_1 = \frac{\sigma_{\min}}{2}$, we obtain that:
\begin{equation*}
\begin{split}
\|\tilde{r}_p(\alpha),\bm{\tilde{R}}_d(\alpha)\|_2 \leq \frac{\mu_k(\alpha)}{\mu_0} K_N,\ \text{for all}\ \alpha \in \bigg(0, \min\big\{\alpha^*,\frac{\sigma_{\min} K_N}{4\xi_2 \mu_0}\big\}\bigg].
\end{split}
\end{equation*}
\noindent Similarly, 
\begin{equation*}
\begin{split}
\|\tilde{r}_p(\alpha),\bm{\tilde{R}}_d(\alpha)\|_{\mathcal{S}} \leq \frac{\mu_k(\alpha)}{\mu_0} \gamma_{\mathcal{S}} \rho,\ \text{for all}\ \alpha \in \bigg(0, \min \big\{\alpha^*,\frac{\sigma_{\min} \gamma_{\mathcal{S}} \rho}{4\xi_{\mathcal{S}} \mu_0}\big\}\bigg].
\end{split}
\end{equation*}
\par Hence, we have that:
\begin{equation} \label{Lemma step-length, STEPLENGTH BOUND}
\bar{\alpha} \coloneqq \min \bigg\{\alpha^*,\frac{\sigma_{\min} K_N}{4\xi_2 \mu_0}, \frac{\sigma_{\min} \gamma_{\mathcal{S}} \rho}{4\xi_{\mathcal{S}} \mu_0} \bigg\}.
\end{equation}
\noindent Since $\bar{\alpha} = \Omega\big(\frac{1}{n^{4}}\big)$, we know that there must exist a constant $\bar{\kappa} > 0$, independent of $n$, $m$ and of the iteration $k$, such that $\bar{\alpha} \geq \frac{\kappa}{n^{4}}$, for all $k \geq 0$, and this completes the proof. 
\end{proof}
\noindent The following theorem summarizes our results.
\begin{theorem} \label{Theorem mu convergence}
Given Assumptions \textnormal{\ref{Assumption 1}, \ref{Assumption 2}}, the sequence $\{\mu_k\}$ generated by Algorithm \textnormal{\ref{Algorithm PMM-IPM}} converges Q-linearly to zero, and the sequences of regularized residual norms 
$$\big\{\|A\bm{X}_k + \mu_k  (y_k-\lambda_k) - b-\frac{\mu_k}{\mu_0}\bar{b}\|_2\big\}\ \text{and}\  \big\{\| A^\top  y_k + \bm{Z}_k - \mu_k (\bm{X}_k - \bm{\Xi}_k) - \bm{C} - \frac{\mu_k}{\mu_0}\bm{\bar{C}}\|_2\big\}$$
converge R-linearly to zero.
\end{theorem}
\begin{proof}
\noindent From (\ref{Lemma step-length relation 3}) we have that:
$$ \mu_{k+1} \leq (1-0.01\alpha_k)\mu_k,$$
\noindent while, from (\ref{Lemma step-length, STEPLENGTH BOUND}), we know that $\forall\ k \geq 0$, $\exists\ \bar{\alpha} \geq \frac{\bar{\kappa}}{n^4}$ such that $\alpha_k \geq \bar{\alpha}$. Hence, we can easily see that $\mu_k \rightarrow 0$. On the other hand, from the neighbourhood conditions, we know that for all $k \geq 0$:
$$ \bigg\|A\bm{X}_k + \mu_k (y_k-\lambda_k) - b - \frac{\mu_k}{\mu_0}\bar{b}\bigg\|_2 \leq K_N \frac{\mu_k}{\mu_0}$$
\noindent and
$$\bigg\| A^\top  y_k + \bm{Z}_k - \mu_k (\bm{X}_k - \bm{\Xi}_k) - \bm{C}- \frac{\mu_k}{\mu_0}\bm{\bar{C}}\bigg\|_2 \leq K_N \frac{\mu_k}{\mu_0}.$$
\noindent This completes the proof.  
\end{proof}
\par The polynomial complexity of Algorithm \ref{Algorithm PMM-IPM} is established in the following theorem.
\begin{theorem} \label{Theorem complexity}
\noindent Let $\varepsilon \in (0,1)$ be a given error tolerance. Choose a starting point for Algorithm \textnormal{\ref{Algorithm PMM-IPM}} as in \eqref{starting point}, such that $\mu_0 \leq \frac{K}{\varepsilon^{\omega}}$ for some positive constants $K,\ \omega$. Given Assumptions \textnormal{\ref{Assumption 1}} and \textnormal{\ref{Assumption 2}}, there exists an index $k_0 \geq 0$ with:
$$k_0 = O\bigg(n^{4}\big|\log \frac{1}{\varepsilon}\big|\bigg)$$
\noindent such that the iterates $\{(X_k,y_k,Z_k)\}$ generated from Algorithm \textnormal{\ref{Algorithm PMM-IPM}} satisfy:
$$\mu_k \leq \varepsilon,\ \ \ \ \text{for all}\ k\geq k_0.$$  
\end{theorem}
\begin{proof}
The proof can be found in \cite[Theorem 3.8]{Pougk_Gond_COAP}.
\end{proof}
\par Finally, we present the global convergence guarantee of Algorithm \ref{Algorithm PMM-IPM}.
\begin{theorem} \label{Theorem convergence for the feasible case}
Suppose that Algorithm \textnormal{\ref{Algorithm PMM-IPM}} terminates when a limit point is reached. Then, if Assumptions \textnormal{\ref{Assumption 1}} and \textnormal{\ref{Assumption 2}} hold, every limit point of $\{(X_k,y_k,Z_k)\}$ determines a primal-dual solution of the non-regularized pair \textnormal{(\ref{non-regularized primal})--(\ref{non-regularized dual})}.
\end{theorem}
\begin{proof}
\noindent From Theorem \ref{Theorem mu convergence}, we know that $\{\mu_k\} \rightarrow 0$, and hence, there exists a sub-sequence $\mathcal{K} \subseteq \mathbb{N}$, such that:
$$\{ A\bm{X}_k + \mu_k (y_k - \lambda_k) -b -\frac{\mu_k}{\mu_0}\bar{b}\}_{\mathcal{K}} \rightarrow 0,\ \{ A^\top  y_k + \bm{Z}_k - \mu_k (\bm{X}_k - \bm{\Xi}_k)-\bm{C}-\frac{\mu_k}{\mu_0}\bm{\bar{C}}\}_{\mathcal{K}} \rightarrow 0.$$
\noindent However, since Assumptions \ref{Assumption 1} and \ref{Assumption 2} hold, we know from Lemma \ref{Lemma boundedness of x z} that $\{(X_k,y_k,Z_k)\}$ is a bounded sequence. Hence, we obtain that:
$$\{ A\bm{X}_k - b\}_{\mathcal{K}} \rightarrow 0,\ \{A^\top y_k +\bm{Z}_k -\bm{C}\}_{\mathcal{K}} \rightarrow 0.$$
\noindent One can readily observe that the limit point of the algorithm satisfies the optimality conditions of \eqref{non-regularized primal}--\eqref{non-regularized dual}, since $\langle X_k, Z_k \rangle \rightarrow 0$ and $X_k,\ Z_k \in \mathcal{S}^n_+$.  
\end{proof}

\begin{remark}
As mentioned at the end of Section \textnormal{\ref{section Algorithmic Framework}}, we do not study the conditions under which one can guarantee that $X_k - \Xi_k \rightarrow 0$ and $y_k - \lambda_k \rightarrow 0$, although this could be possible. This is because the method is shown to converge globally even if this is not the case. Indeed, notice that if one were to choose $X_0 = 0$ and $\lambda_0 = 0$, and simply ignore the last conditional statement of Algorithm \textnormal{\ref{Algorithm PMM-IPM}}, the convergence analysis established in this section would still hold. In this case, the method would be interpreted as an interior point-quadratic penalty method, and we could consider the regularization as a diminishing primal-dual Tikhonov regularizer (i.e. a variant of the regularization proposed in \textnormal{\cite{SaundersTomlin_Tech_Rep}}).
\end{remark}

\section{A Sufficient Condition for Strong Duality} \label{section Infeasible problems} 

\par We now drop Assumptions \ref{Assumption 1}, \ref{Assumption 2}, in order to analyze the behaviour of the algorithm when solving problems that are strongly (or weakly) infeasible, problems for which strong duality does not hold (weakly feasible), or problems for which the primal or the dual solution is not attained. For a formal definition and a comprehensive study of the previous types of problems we refer the reader to \cite{Liu_Pataki_MATH_PROG}, and the references therein. Below we provide a well-known result, stating that strong duality holds if and only if there exists a KKT point. 
\begin{Proposition} \label{Prop. KKT and strong duality}
Let \textnormal{\eqref{non-regularized primal}--\eqref{non-regularized dual}} be given. Then, $\textnormal{val\eqref{non-regularized primal}} \geq \textnormal{val\eqref{non-regularized dual}}$, where $\textnormal{val}(\cdot)$ denotes the optimal objective value of a problem. Moreover, $\textnormal{val\eqref{non-regularized primal}} = \textnormal{val\eqref{non-regularized dual}}$ and $(X^*,y^*,Z^*)$ is an optimal solution for \textnormal{\eqref{non-regularized primal}--\eqref{non-regularized dual}}, if and only if  $(X^*,y^*,Z^*)$  satisfies the (KKT) optimality conditions in \eqref{non-regularized F.O.C}.
\end{Proposition}
\begin{proof}
This is a well-known fact, the proof of which can be found in \cite[Proposition 2.1]{ShapSchein_BOOK_SPRINGER}. 
\end{proof}
  Let us employ the following two premises:
\begin{premise} \label{Premise 1}
During the iterations of Algorithm \textnormal{\ref{Algorithm PMM-IPM}}, the sequences $\{\|y_k - \lambda_k\|_2\}$ and $\{\|X_k - \Xi_k\|_F\}$, remain bounded.
\end{premise}

\begin{premise} \label{Premise 2}
There does not exist a primal-dual triple, satisfying the KKT conditions in \eqref{non-regularized F.O.C} associated with the primal-dual pair \textnormal{(\ref{non-regularized primal})--(\ref{non-regularized dual})}.
\end{premise}
\noindent The following analysis extends the result presented in \cite[Section 4]{Pougk_Gond_COAP}, and is based on the developments in \cite[Sections 10 \& 11]{Dehg_Goff_Orban_OMS}. In what follows, we show that Premises \ref{Premise 1} and \ref{Premise 2} are contradictory. In other words, if Premise \ref{Premise 2} holds (which means that strong duality does not hold for the problem under consideration), then Premise \ref{Premise 1} cannot hold, and hence Premise \ref{Premise 1} is a sufficient condition for strong duality (and its negation is a necessary condition for Premise \ref{Premise 2}). We show that if Premise \ref{Premise 1} holds, then the algorithm converges to an optimal solution. If not, however, it does not necessarily mean that the problem under consideration is infeasible. For example, this could happen if either \eqref{non-regularized primal} or \eqref{non-regularized dual} is strongly infeasible, weakly infeasible, and in some cases even if either of the problems is weakly feasible (e.g. see \cite{Liu_Pataki_MATH_PROG,ShapSchein_BOOK_SPRINGER}). As we discuss later, the knowledge that Premise \ref{Premise 1} does not hold could be useful in detecting pathological problems.

\begin{lemma} \label{Lemma infeasibility bounded Newton}
Given Premise \textnormal{\ref{Premise 1}}, and by assuming that $\langle X_k, Z_k \rangle > \varepsilon$, for some $\varepsilon >0$, for all iterations $k$ of Algorithm \textnormal{\ref{Algorithm PMM-IPM}}, the Newton direction produced by \textnormal{(\ref{inexact vectorized Newton System})} is uniformly bounded by a constant dependent only on $n$ and/or $m$. 
\end{lemma}
\begin{proof}
\par The proof is omitted since it follows exactly the developments in \cite[Lemma 10.1]{Dehg_Goff_Orban_OMS}. We notice that the regularization terms (blocks (1,1) and (2,2) in the Jacobian  matrix in \eqref{inexact vectorized Newton System}) depend on $\mu_k$ which by assumption is always bounded away 
from zero: $\mu_k \geq \frac{\epsilon}{n}$. 

 \end{proof}
 
\par In the following Lemma, we prove by contradiction that the parameter $\mu_k$ of Algorithm \ref{Algorithm PMM-IPM} converges to zero, given that Premise \ref{Premise 1} holds. 

\begin{lemma} \label{Lemma infeasibility mu to zero}
Given Premise \textnormal{\ref{Premise 1}}, and a sequence $(X_k,y_k,Z_k) \in \mathcal{N}_{\mu_k}(\Xi_k,\lambda_k)$ produced by Algorithm \textnormal{\ref{Algorithm PMM-IPM}}, the sequence $\{\mu_k\}$ converges to zero.
\end{lemma}
\begin{proof}
\noindent Assume, by virtue of contradiction, that $\mu_k > \varepsilon > 0$, $\text{for all}\ k \geq 0$. Then, we know (from Lemma \ref{Lemma infeasibility bounded Newton}) that the Newton direction obtained by the algorithm at every iteration, after solving (\ref{inexact vectorized Newton System}), will be uniformly bounded by a constant dependent only on $n$, that is, there exists a positive constant $K^{\dagger}$, such that $\|(\Delta X_k,\Delta y_k,\Delta Z_k)\|_2 \leq K^{\dagger}$. We define $\tilde{r}_p(\alpha)$ and  $\bm{\tilde{R}}_d(\alpha)$ as in \eqref{primal infeasibility vector} and \eqref{dual infeasibility matrix}, respectively, for which we know that equalities \eqref{primal infeasibility formula} and \eqref{dual infeasibility formula} hold, respectively.
 Take any $k \geq 0$ and define the following functions:
\begin{equation*}
\begin{split}
f_1(\alpha) \coloneqq \ &\langle X_k(\alpha), Z_k(\alpha)\rangle - (1 -\alpha(1-\frac{\sigma_{\min}}{2}))\langle X_k, Z_k\rangle,\\
f_2(\alpha) \coloneqq \ & \gamma_{\mu}\mu_k(\alpha) - \|H_{{P_k}(\alpha)}(X_k(\alpha)Z_k(\alpha)) - \mu_k(\alpha)\|_F \\
f_3(\alpha) \coloneqq \ & (1-0.01\alpha)\langle X_k,Z_k\rangle - \langle X_k(\alpha), Z_k(\alpha)\rangle,\\
g_{2}(\alpha) \coloneqq \ & \frac{\mu_k(\alpha)}{\mu_0}K_N - \|(\tilde{r}_p(\alpha),\bm{\tilde{R}}_d(\alpha))\|_2, 
\end{split}
\end{equation*}
\noindent where $\mu_k(\alpha) = \frac{\langle X_k + \alpha \Delta X_k, Z_k + \alpha \Delta Z_k\rangle}{n}$, $(X_k(\alpha),y_k(\alpha),Z_k(\alpha)) = (X_k + \alpha \Delta X_k,y_k + \alpha \Delta y_k,Z_k + \alpha \Delta Z_k)$. We would like to show that there exists $\alpha^* > 0$, such that:
$$f_1(\alpha) \geq 0,\quad f_2(\alpha) \geq 0,\quad f_3(\alpha) \geq 0,\quad g_2(\alpha) \geq 0,\ \text{for all}\ \alpha \in (0,\alpha^*].$$
\noindent These conditions model the requirement that the next iteration of Algorithm \ref{Algorithm PMM-IPM} must lie in the updated neighbourhood $\mathcal{N}_{\mu_{k+1}}(\Xi_k,\lambda_{k})$ (notice however that the restriction with respect to the semi-norm defined  in \eqref{semi-norm definition} is not required here, and indeed it cannot be incorporated unless $\textnormal{rank}(A) = m$). Since Algorithm \ref{Algorithm PMM-IPM} updates the parameters $\lambda_k,\ \Xi_k$ only if the selected new iterate belongs to the new neighbourhood, defined using the updated parameters (again, ignoring the restrictions with respect to the semi-norm), it suffices to show that $(X_{k+1},y_{k+1},Z_{k+1}) \in \mathcal{N}_{\mu_{k+1}}(\Xi_k,\lambda_{k})$. 
\par Proving the existence of $\alpha^* > 0$, such that each of the aforementioned functions is positive, follows exactly the developments in Lemmas \ref{Lemma step-length-part 1}--\ref{Lemma step-length-part 2}, with the only difference being that the bounds on the directions are not explicitly specified in this case. Using the same methodology as in aforementioned lemmas, while keeping in mind our assumption, namely $\langle X_k, Z_k \rangle> \varepsilon$, we can show that:

\begin{equation} \label{infeasibility minimum step-length}
\alpha^* \coloneqq \min \bigg\{1,\frac{\sigma_{\min}\epsilon}{2(K^{\dagger})^2}, \frac{(1-\gamma_{\mu})\sigma_{\min}{\gamma_{\mu}\epsilon}}{2n(K^{\dagger})^2},\frac{0.49\epsilon}{2(K^{\dagger})^2}, \frac{\sigma_{\min}K_N \epsilon}{4\mu_0(\xi_2)} \bigg\},
\end{equation}
\noindent where $\xi_2$ is a bounded constant, defined as in \eqref{step-length, auxiliary constants}, and dependent on $K^{\dagger}$. However, using the inequality:
$$\mu_{k+1} \leq (1-0.01 \alpha)\mu_k,\ \text{for all}\ \alpha \in [0,\alpha^*]$$
\noindent we get that $\mu_k \rightarrow 0$, which contradicts our assumption that $\mu_k > \varepsilon,\ \forall\ k\geq 0$, and completes the proof.  
\end{proof}
\noindent Finally, using the following Theorem, we derive a necessary condition for lack of strong duality.
\begin{theorem} \label{Theorem Infeasibility condition}
Given Premise \textnormal{\ref{Premise 2}}, i.e. there does not exist a KKT triple for the pair \textnormal{(\ref{non-regularized primal})--(\ref{non-regularized dual})}, then Premise \textnormal{\ref{Premise 1}} fails to hold.
\end{theorem}
\begin{proof}
\noindent By virtue of contradiction, let Premise \ref{Premise 1} hold. In Lemma \ref{Lemma infeasibility mu to zero}, we proved that given Premise \ref{Premise 1}, Algorithm \ref{Algorithm PMM-IPM} produces iterates that belong to the neighbourhood (\ref{Small neighbourhood}) and $\mu_k \rightarrow 0$. But from the neighbourhood conditions we can observe that:
$$\|A\bm{X}_k + \mu_k(y_k - \lambda_k) - b - \frac{\mu_k}{\mu_0}\bar{b} \|_2 \leq K_N\frac{\mu_k}{\mu_0},$$
\noindent and
$$\|A^\top y_k + \bm{Z}_k - \mu_k(\bm{X}_k - \bm{\Xi}_k)-\bm{C}-\frac{\mu_k}{\mu_0}\bm{\bar{C}}\|_2 \leq K_N \frac{\mu_k}{\mu_0}.$$
\noindent Hence, we can choose a sub-sequence $\mathcal{K} \subseteq \mathbb{N}$, for which:
$$\{A\bm{X}_k + \mu_k(y_k - \lambda_k) - b - \frac{\mu_k}{\mu_0}\bar{b} \}_{\mathcal{K}} \rightarrow 0,\ \text{and} \ \{A^\top y_k + \bm{Z}_k - \mu_k(\bm{X}_k - \bm{\Xi}_k)-\bm{C}-\frac{\mu_k}{\mu_0}\bm{\bar{C}}\}_{\mathcal{K}} \rightarrow 0.$$
\noindent But since $\|y_k-\lambda_k\|_2$ and $\|X_k - \Xi_k\|_F$ are bounded, while $\mu_k \rightarrow 0$, we have that:
$$\{A\bm{X}_k - b\}_{\mathcal{K}} \rightarrow 0,\ \{\bm{C} - A^\top  y_k - \bm{Z}_k\}_{\mathcal{K}} \rightarrow 0,\ \text{and}\ \{\langle X_k, Z_k\rangle\}_{\mathcal{K}} \rightarrow 0.$$
\noindent This contradicts Premise \ref{Premise 2}, i.e. that the pair (\ref{non-regularized primal})--(\ref{non-regularized dual}) does not have a KKT triple, and completes the proof.  
\end{proof}

\par In the previous Theorem, we proved that the negation of Premise \ref{Premise 1} is a necessary condition for Premise \ref{Premise 2}. Nevertheless, this does not mean that the condition is also sufficient. In order to obtain a more reliable algorithmic test for lack of strong duality, we have to use the properties of Algorithm \ref{Algorithm PMM-IPM}. In particular, we can notice that if there does not exist a KKT point, then the PMM sub-problems will stop being updated after a finite number of iterations. In that case, we know from Theorem \ref{Theorem Infeasibility condition} that the sequence $\|(\bm{X}_k- \bm{\Xi}_k,y_k -\lambda_k)\|_2$ will grow unbounded. Hence, we can define a maximum number of iterations per PMM sub-problem, say $k_{\dagger} > 0$, as well as a very large constant $K_{\dagger}$. Then, if $\|(\bm{X}_k- \bm{\Xi}_k,y_k -\lambda_k)\|_2 >  K_{\dagger}$ and $k_{in} \geq k_{\dagger}$ (where $k_{in}$ counts the number of IPM iterations per PMM sub-problem), the algorithm is terminated with a guess that there does not exist a KKT point for \eqref{non-regularized primal}--\eqref{non-regularized dual}.

\begin{remark}
Let us notice that the analysis in Section \textnormal{\ref{section Polynomial Convergence}} employs the standard assumptions used when analyzing a non-regularized IPM. However, the method could still be useful if these assumptions were not met. Indeed, if for example the constraint matrix was not of full row rank, one could still prove global convergence of the method, using the methodology employed in this section by assuming that Premise \textnormal{\ref{Premise 1}} holds and Premise \textnormal{\ref{Premise 2}} does not (or by following the developments in \textnormal{\cite{Dehg_Goff_Orban_OMS}}). Furthermore, in practice the method would not encounter any numerical issues with the inversion of the Newton system (see \textnormal{\cite{ArmandBenoist_MATH_PROG}}). Nevertheless, showing polynomial complexity in this case is still an open problem. The aim of this work is to show that under the standard Assumptions \textnormal{\ref{Assumption 1}, \ref{Assumption 2}}, Algorithm \ref{Algorithm PMM-IPM} is able to enjoy polynomial complexity, while having to solve better conditioned systems than those solved by standard IPMs at each iteration, thus ensuring better stability (and as a result better robustness and potentially better efficiency).
\end{remark}

\section{Conclusions} \label{section conclusions}
\par 
In this paper we developed and analyzed an interior point-proximal method of multipliers, suitable for solving linear positive semi-definite programs, without requiring the exact solution of the associated Newton systems. By generalizing appropriately some previous results on convex quadratic programming, we show that IP-PMM inherits the polynomial complexity of standard non-regularized IPM schemes when applied to SDP problems, under standard assumptions, while having to approximately solve better-conditioned Newton systems, compared to their non-regularized counterparts. Furthermore, we provide a tuning for the proximal penalty parameters based on the well-studied barrier parameter, which can be used to guide any subsequent implementation of the method. Finally, we study the behaviour of the algorithm when applied to problems for which no KKT point exists, and give a necessary condition which can be used to construct detection mechanisms for identifying such pathological cases. 
\par 
A future research direction would be to construct a robust and efficient implementation of the method, which should utilize some Krylov subspace solver alongside an appropriate preconditioner for the solution of the associated linear systems. Given previous implementations of IP-PMM for other classes of problems appearing in the literature, we expect that the theory can successfully guide the implementation, yielding a competitive, robust, and efficient method.

\bibliography{references} 
\bibliographystyle{abbrvdin}

\end{document}